\documentclass[11pt]{amsart}
\usepackage[top=1.5in, bottom=1.5in, left=1.5in, right=1.5in]{geometry}
\usepackage{amsmath}
\usepackage{amssymb}
\usepackage{thmtools}
\usepackage{hyperref}
\usepackage{amsthm}
\usepackage{mathtools}
\usepackage[capitalize,noabbrev]{cleveref}
\usepackage{xcolor}
\usepackage{rotating}
\usepackage{tikz,tikz-cd}
\usetikzlibrary{arrows.meta,calc,positioning,shapes.geometric,decorations.markings}
\usepackage{ytableau}
\ytableausetup{boxsize=1em}

\allowdisplaybreaks

\numberwithin{equation}{section}

\newtheorem{thm}{Theorem}[section]
\newtheorem{lemma}[thm]{Lemma}

\newtheorem{conj}[thm]{Conjecture}

\theoremstyle{definition}

\newtheorem{definition}[thm]{Definition}
\newtheorem{example}[thm]{Example}
\newtheorem{remark}[thm]{Remark}

\crefname{thm}{Theorem}{Theorems}
\crefname{lemma}{Lemma}{Lemmas}
\crefname{cor}{Corollary}{Corollaries}
\crefname{prop}{Proposition}{Propositions}
\crefname{conj}{Conjecture}{Conjectures}

\crefname{definition}{Definition}{Definitions}
\crefname{example}{Example}{Examples}
\crefname{remark}{Remark}{Remarks}
\crefname{question}{Question}{Questions}
\crefname{problem}{Problem}{Problems}

\newcommand{\emailhref}[1]{\email{\href{#1}{#1}}}
\newcommand{\dfn}[1]{\textcolor{blue}{\emph{#1}}}

\newcommand{\BK}{\mathsf{BK}}
\newcommand{\dom}{\mathsf{dom}}

\title[Cyclic Sieving for Staircase Plane Partitions]{Cyclic Sieving for Staircase Plane Partitions \\ via Crystals and Electrical Networks}
\author{Sam Hopkins}\emailhref{samuelfhopkins@gmail.com}
\address{Department of Mathematics, Howard University, Washington, DC, USA}
\author{Jesse Kim}\emailhref{jesse.kim@ufl.edu}
\address{Department of Mathematics, University of Florida, Gainesville, FL, USA}
\author{Stephan Pfannerer}\emailhref{stephan@pfannerer-mittas.net}
\address{Dept.~of Combinatorics and Optimization, University of Waterloo, ON, Canada}

\begin{document}

\begin{abstract}
We prove a cyclic sieving result for the action of promotion on the staircase plane partitions of height two. Our proof has two major algebraic inputs: an interpretation of this promotion action in terms of tensor powers of the spin crystal that was recently studied by Pappe--Pfannerer--Schilling--Simone, and the bush basis of the degree two part of the coordinate ring of the space of electrical networks that was recently introduced by Gao--Lam--Xu. Moreover, we explain how the existence of an electrical canonical basis in all degrees would yield cyclic sieving for promotion of staircase plane partitions of all heights.
\end{abstract}

\maketitle

\section{Introduction} \label{sec:intro}

Let $P$ be a finite poset. For an integer $m \geq 1$, a \dfn{$P$-partition of height $m$} is an order-preserving map $P\to \{0,1,\ldots,m\}$. We use $\mathcal{PP}^m(P)$ to denote the set of $P$-partitions of height $m$, and define the \dfn{order polynomial} $\Omega_P(m)$ of $P$ to be the polynomial in~$m$ for which $\Omega_P(m)=\#\mathcal{PP}^m(P)$.\footnote{This differs from the usual definition of the order polynomial, as in \cite[\S3]{stanley2012ec1}, by a shift by one, but it leads to cleaner formulas.}

\dfn{(Piecewise-linear) rowmotion} is an invertible operator acting on the order polytope of a finite poset~$P$ which was introduced about ten years ago by Einstein and Propp~\cite{einstein2014piecewiselinear,einstein2021combinatorial} and has subsequently received significant attention. For any $m \geq 1$, by restricting to the rational points in the order polytope with denominator dividing $m$, we can consider (piecewise-linear) rowmotion as an invertible operator $\mathrm{Row}\colon \mathcal{PP}^m(P) \to \mathcal{PP}^m(P)$ acting on height $m$ $P$-partitions. While this rowmotion action is defined for any finite poset, it is only for special families of posets that it exhibits good behavior. In~\cite{hopkins2024order}, the first author put forward the following meta-conjecture concerning when rowmotion of $P$-partitions behaves well. 

\begin{conj}[{\cite{hopkins2024order}}] \label{conj:main}
Let $P$ be a finite graded poset of rank $r$ for which the roots of $\Omega_P(m)$ are all integers or half-integers, and in the latter case possessing an order two automorphism. Let $m \geq 1$ be an integer and define
\[\Omega_P(m;q) := \prod_{\alpha} \frac{1-q^{\kappa(m-\alpha)}}{1-q^{-\kappa \alpha}},\]
a product over all the roots $\alpha$ of $\Omega_P(m)$ with multiplicity, and  with $\kappa := 1$ if the roots are integers or $\kappa := 2$ if they are half-integers. Then~$\Omega_P(m;q)$ is a polynomial in $q$ with nonnegative integer coefficients, and $(\mathcal{PP}^m(P),\langle \mathrm{Row} \rangle \simeq \mathbb{Z}/\kappa(r+2)\mathbb{Z},\Omega_P(m;q))$ exhibits \dfn{cyclic sieving}.
\end{conj}

For background on the \dfn{cyclic sieving phenomenon} (CSP), consult~\cite{reiner2004cyclic,sagan2011cyclic}. Note in particular that when we have a CSP where the sieving polynomial has a product formula as a rational function, \emph{every} symmetry class under the action is enumerated by a product formula. So \cref{conj:main} says that rowmotion of $P$-partitions for these special families of posets $P$ is very well behaved indeed.

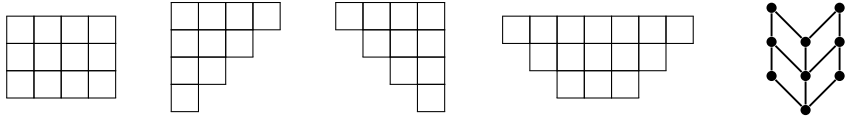
\begin{figure}
\begin{center}\scalebox{0.9}{
\parbox{0.75in}{\begin{center}\ydiagram{4,4,4}\end{center}} \quad \parbox{0.75in}{\begin{center}\ydiagram{4,3,2,1}\end{center}} \quad \parbox{0.75in}{\begin{center}\ydiagram{4,1+3,2+2,3+1}\end{center}} \quad \parbox{1.25in}{\begin{center}\ydiagram{7,1+5,2+3}\end{center}} \quad \parbox{0.75in}{\begin{center}\begin{tikzpicture}[scale=0.5]
        \node at (0,2) {};
	\node[shape=circle,fill=black,inner sep=1.5] (A1) at (0,-1) {};
	\node[shape=circle,fill=black,inner sep=1.5] (B1) at (-1,0) {};
	\node[shape=circle,fill=black,inner sep=1.5] (C1) at (1,0) {};
	\node[shape=circle,fill=black,inner sep=1.5] (A2) at (0,0) {};
	\node[shape=circle,fill=black,inner sep=1.5] (B2) at (-1,1) {};
	\node[shape=circle,fill=black,inner sep=1.5] (C2) at (1,1) {};
	\node[shape=circle,fill=black,inner sep=1.5] (A3) at (0,1) {};
	\node[shape=circle,fill=black,inner sep=1.5] (B3) at (-1,2) {};
	\node[shape=circle,fill=black,inner sep=1.5] (C3) at (1,2) {};
	\draw[thick] (B1)--(A1);
	\draw[thick] (C1)--(A1);
	\draw[thick] (A1)--(A2);
	\draw[thick] (B1)--(B2);
	\draw[thick] (C1)--(C2);
	\draw[thick] (B2)--(A2);
	\draw[thick] (C2)--(A2);
	\draw[thick] (A2)--(A3);
	\draw[thick] (B2)--(B3);
	\draw[thick] (C2)--(C3);
	\draw[thick] (B3)--(A3);
	\draw[thick] (C3)--(A3);
\end{tikzpicture}\end{center}}}
\end{center}
\caption{Examples of the families of posets to which \cref{conj:main} applies.
} \label{fig:exs}
\end{figure}

Families of posets to which \cref{conj:main} applies include the \dfn{rectangles}, \dfn{staircases}, \dfn{shifted staircases}, \dfn{trapezoids}, and \dfn{chains of $V$'s}. Examples of these families of posets are depicted in \cref{fig:exs}. Note that rectangles and staircases are Young diagram shapes, shifted staircases and trapezoids are shifted shapes, and chains of~$V$'s are neither. Also, $\kappa=2$ for staircases and chains of $V$'s, and $\kappa=1$ for the others.

For $P$ a \dfn{root poset} (including staircases and maximal trapezoids), the case $m=1$ of \cref{conj:main} was proved by Armstrong, Stump, and Thomas~\cite{armstrong2013uniform}. For $P$ a \dfn{minuscule poset} (including rectangles and shifted staircases), the case $m=1$ of \cref{conj:main} was proved by Rush and Shi~\cite{rush2013orbits}. For all $m \geq 1$ and $P$ a rectangle, \cref{conj:main} was essentially proved by Rhoades~\cite{rhoades2010cyclic} (see also~\cite{hopkins2020cyclic}). Johnson and Liu~\cite{johnson2023plane} established an equivalence of rowmotion for rectangles and trapezoids, which combined with Rhoades's result shows that \cref{conj:main} is true also for all~$m\geq 1$ when $P$ is trapezoid.

As far as we know, those are all cases of \cref{conj:main} which have previously been proved. See~\cite[\S5]{hopkins2024order} for more details and references. Here we prove a new case of \cref{conj:main}. Let $\delta_n=(n,n-1,\ldots,1)$ denote the staircase Young diagram shape poset. We have, for integers $n \geq 2$ and $m \geq 1$, that
\begin{equation} \label{eqn:q-multi-cat}
    \Omega_{\delta_n}(m;q) = \prod_{1\leq i \leq j \leq n} \frac{1-q^{i+j+2m}}{1-q^{i+j}}
\end{equation}
which is known to be a polynomial in $q$ with nonnegative integer coefficients by work of Proctor~\cite{proctor1988odd, proctor1990symmetric}. \Cref{conj:main} for $P=\delta_n$ a staircase was also conjectured independently by Pappe et al.~in \cite[Conjecture~4.23]{pappe2024promotion}.

\begin{remark}
The polynomial $\Omega_{\delta_n}(m;q)$ is often called the \dfn{$q$-multi-Catalan number}, and it is also conjectured that this same polynomial is a cyclic sieving polynomial for the action of rotation on ``multi-triangulations''~\cite{serrano2012maximal,ceballos2014subword}. However, the orders of rotation of multi-triangulations and of rowmotion of $P$-partitions for $P=\delta_n$ are not the same (even for $m=1$), so we do not know any direct connection between these two CSP conjectures.
\end{remark}

As mentioned, for $P=\delta_n$, the case $m=1$ of \cref{conj:main} was proved in~\cite{armstrong2013uniform}. Our main theorem addresses the case of $m=2$:

\begin{thm} \label{thm:main}
\Cref{conj:main} is true when $P = \delta_n$ is a staircase and $m=2$.
\end{thm}

A few terminological remarks concerning \cref{thm:main} are in order. First of all, when $P = a \times b$ is a rectangle, $P$-partitions are called \dfn{plane partitions}. So for~$P=\delta_n$ a staircase, we refer to $P$-partitions as \dfn{staircase plane partitions}. Also, rather than consider rowmotion, which is a composition of \dfn{toggles} (local involutions) from ``top-to-bottom,'' in what follows we instead work with \dfn{promotion}, which is the composition of these toggles from ``left-to-right.'' Pioneering work of Striker and Williams~\cite{striker2012promotion} showed that rowmotion and promotion are always conjugate, so their orbit structures are the same. Hence our main result, \cref{thm:main}, can be restated as a cyclic sieving result for promotion of staircase plane partitions of height two. Let us now briefly explain our approach to understanding promotion of staircase plane partitions.

It is helpful to first review the proof of Rhoades's CSP result for promotion of usual (i.e., rectangular) plane partitions. One way to interpret what Rhoades did in~\cite{rhoades2010cyclic} is as follows (see also~\cite{hopkins2020cyclic} and especially~\cite[Thm.~2.2]{lam2019cyclic}). Fix $a, b \geq 1$ and let $P=a\times b$ be the $a \times b$ rectangle, so that~$\mathcal{PP}^m(a\times b)$ is the set of $a \times b$ plane partitions of height~$m$. The degree~$m$ part of the homogeneous \dfn{coordinate ring} of the \dfn{Grassmannian} $\mathrm{Gr}(a,a+b)$ of $a$-planes in~$\mathbb{C}^{a+b}$ has a basis indexed by the elements of $\mathcal{PP}^m(a\times b)$. In fact, it has many such bases, but a particularly nice one is the Lusztig/Kashiwara \dfn{dual canonical basis}. Also, $\mathrm{Gr}(a,a+b)$ carries an action of the general linear group $\mathrm{GL}(a+b)$, and so its coordinate ring does as well. Rhoades showed that a particular lift of the long cycle in the symmetric group $S_{a+b}$ to $\mathrm{GL}(a+b)$ acts as promotion of plane partitions on the dual canonical basis. Then a standard character computation yields~\cref{conj:main} in this case.

Our proof of~\cref{thm:main} follows a similar structure, where we realize the action of promotion algebraically in terms of a nice basis of the coordinate ring of a homogeneous space. The space in question is the \dfn{space of electrical networks}~\cite{lam2018electroid,bychkov2023electrical,chepuri2026electrical}. In~\cite{gao2025electrical}, the coordinate ring of the space of electrical networks is termed the \dfn{grove algebra}. The basis we use to prove our CSP is a basis of the degree two part of the grove algebra called the \dfn{bush basis}, which was recently introduced by Gao, Lam, and Xu~\cite[\S4]{gao2025electrical}.

In order to relate promotion of height two staircase plane partitions to the bush basis, we use another algebraic incarnation of promotion. Pfannerer, Rubey and Westbury~\cite{pfannerer2020promotion} explained how the Henriques--Kamnitzer~\cite{henriques2006crystals} cactus group action on the \dfn{tensor product} of \dfn{crystals} gives a promotion operator on these tensor products and their invariants. Pappe, Pfannerer, Schilling, and Simone~\cite{pappe2024promotion} further explained how, in the particular case of tensor powers of the \dfn{spin representation} of a spin group, this crystal promotion operator is promotion of staircase plane partitions. Moreover, because of the exceptional isomorphism of the rank two Lie algebras~$B_2$ and $C_2$, this crystal interpretation also allows us to show that promotion of height two staircase plane partitions is the same as \dfn{rotation} of \dfn{$3$-noncrossing perfect matchings}. These matchings index the bush basis, and rotation corresponds to a natural cyclic action on the basis. Together with a standard character computation for a symplectic group representation, this completes our proof.

\Cref{fig:proof-overview} gives an outline of the proof of \cref{thm:main}, including the bijections, cyclic actions, and later definitions or results used for each compatibility.

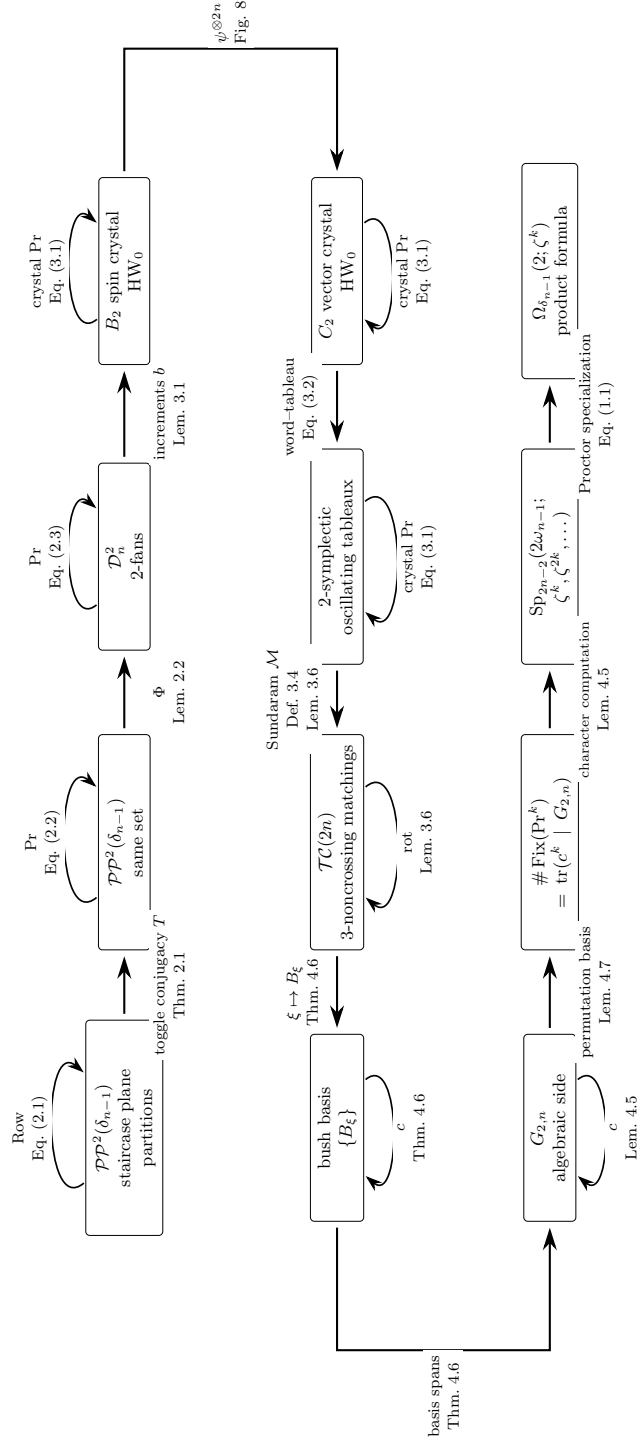
\begin{sidewaysfigure}[p]
\centering
\vspace*{0.5\textwidth}
\resizebox{0.95\textheight}{!}{
\begin{tikzpicture}[
  x=5.05cm,
  y=2.15cm,
  obj/.style={draw,rounded corners=2pt,align=center,text width=3.05cm,minimum height=.9cm,font=\footnotesize},
  wide/.style={obj,text width=3.55cm},
  edge/.style={-{Stealth[length=4mm,width=2.5mm]},line width=1.05pt,shorten >=3pt,shorten <=3pt},
  act/.style={-{Stealth[length=3mm,width=2mm]},line width=.85pt,shorten >=1pt,shorten <=1pt},
  lab/.style={font=\scriptsize,align=center,fill=white,inner sep=2.2pt}
]
\node[wide] (row) at (0,0) {$\mathcal{PP}^2(\delta_{n-1})$\\staircase plane partitions};
\node[wide] (pr) at (1,0) {$\mathcal{PP}^2(\delta_{n-1})$\\same set};
\node[obj] (fan) at (2,0) {$\mathcal D^2_n$\\$2$-fans};
\node[obj] (b) at (3,0) {$B_2$ spin crystal\\$\operatorname{HW}_0$};

\node[obj] (c) at (3,-1.75) {$C_2$ vector crystal\\$\operatorname{HW}_0$};
\node[wide] (osc) at (2,-1.75) {$2$-symplectic\\oscillating tableaux};
\node[wide] (mat) at (1,-1.75) {$\mathcal{TC}(2n)$\\$3$-noncrossing matchings};
\node[obj] (basis) at (0,-1.75) {bush basis\\$\{B_\xi\}$};

\node[obj] (alg) at (0,-3.5) {$G_{2,n}$\\algebraic side};
\node[wide] (trace) at (1,-3.5) {$\#\operatorname{Fix}(\mathrm{Pr}^k)$\\$=\operatorname{tr}(c^k\mid G_{2,n})$};
\node[wide] (char) at (2,-3.5) {$\mathrm{Sp}_{2n-2}(2\omega_{n-1};$\\$\zeta^k,\zeta^{2k},\dots)$};
\node[wide] (prod) at (3,-3.5) {$\Omega_{\delta_{n-1}}(2;\zeta^k)$\\product formula};

\draw[act] ($(row.north west)!0.25!(row.north east)$) to[out=130,in=50,looseness=1.35] node[lab,above] {$\mathrm{Row}$\\Eq.~(\ref{eq:rowmotion-definition})} ($(row.north west)!0.75!(row.north east)$);
\draw[act] ($(pr.north west)!0.25!(pr.north east)$) to[out=130,in=50,looseness=1.35] node[lab,above] {$\mathrm{Pr}$\\Eq.~(\ref{eq:promotion-definition})} ($(pr.north west)!0.75!(pr.north east)$);
\draw[act] ($(fan.north west)!0.25!(fan.north east)$) to[out=130,in=50,looseness=1.35] node[lab,above] {$\mathrm{Pr}$\\Eq.~(\ref{eq:promotion-diag})} ($(fan.north west)!0.75!(fan.north east)$);
\draw[act] ($(b.north west)!0.25!(b.north east)$) to[out=130,in=50,looseness=1.35] node[lab,above] {crystal $\mathrm{Pr}$\\Eq.~(\ref{eq:crystal commutor})} ($(b.north west)!0.75!(b.north east)$);
\draw[act] ($(c.south west)!0.75!(c.south east)$) to[out=-50,in=-130,looseness=1.35] node[lab,below] {crystal $\mathrm{Pr}$\\Eq.~(\ref{eq:crystal commutor})} ($(c.south west)!0.25!(c.south east)$);
\draw[act] ($(osc.south west)!0.75!(osc.south east)$) to[out=-50,in=-130,looseness=1.35] node[lab,below] {crystal $\mathrm{Pr}$\\Eq.~(\ref{eq:crystal commutor})} ($(osc.south west)!0.25!(osc.south east)$);
\draw[act] ($(mat.south west)!0.75!(mat.south east)$) to[out=-50,in=-130,looseness=1.35] node[lab,below] {$\operatorname{rot}$\\Lem.~\ref{lem:map to matchings}} ($(mat.south west)!0.25!(mat.south east)$);
\draw[act] ($(basis.south west)!0.75!(basis.south east)$) to[out=-50,in=-130,looseness=1.35] node[lab,below] {$c$\\Thm.~\ref{thm:bushbasisexpansion}} ($(basis.south west)!0.25!(basis.south east)$);
\draw[act] ($(alg.south west)!0.75!(alg.south east)$) to[out=-50,in=-130,looseness=1.35] node[lab,below] {$c$\\Lem.~\ref{lem:char_comp}} ($(alg.south west)!0.25!(alg.south east)$);

\draw[edge] (row) -- node[lab,below=12pt] {toggle conjugacy $T$\\Thm.~\ref{thm:striker_williams}} (pr);
\draw[edge] (pr) -- node[lab,below=12pt] {$\Phi$\\Lem.~\ref{lem:pp_dyck_bij}} (fan);
\draw[edge] (fan) -- node[lab,below=12pt] {increments $b$\\Lem.~\ref{lem:promotion_crystal}} (b);
\coordinate (bc-turn) at ($(b.east)+(.45,0)$);
\coordinate (bc-mid) at ($(b.east)!0.5!(c.east)+(.45,0)$);
\draw[edge] (b.east) -- (bc-turn) |- (c.east);
\node[lab,right] at (bc-mid) {$\psi^{\otimes 2n}$\\Fig.~\ref{fig:B2-C2-crystal-isomorphism}};

\draw[edge] (c) -- node[lab,above=7pt] {word--tableau\\Eq.~(\ref{eq:word-tableau-dictionary})} (osc);
\draw[edge] (osc) -- node[lab,above=7pt] {Sundaram $\mathcal M$\\Def.~\ref{def:sundaram-map}\\Lem.~\ref{lem:map to matchings}} (mat);
\draw[edge] (mat) -- node[lab,above=7pt] {$\xi\mapsto B_\xi$\\Thm.~\ref{thm:bushbasisexpansion}} (basis);
\coordinate (ba-turn) at ($(basis.west)+(-.45,0)$);
\coordinate (ba-mid) at ($(basis.west)!0.5!(alg.west)+(-.45,0)$);
\draw[edge] (basis.west) -- (ba-turn) |- (alg.west);
\node[lab,left] at (ba-mid) {basis spans\\Thm.~\ref{thm:bushbasisexpansion}};

\draw[edge] (alg) -- node[lab,below=12pt] {permutation basis\\Lem.~\ref{lem:bush}} (trace);
\draw[edge] (trace) -- node[lab,below=12pt] {\tiny character computation\\Lem.~\ref{lem:char_comp}} (char);
\draw[edge] (char) -- node[lab,below=12pt] {Proctor specialization\\Eq.~(\ref{eqn:q-multi-cat})} (prod);
\end{tikzpicture}
}
\caption{Outline of the proof of \cref{thm:main}.}
\label{fig:proof-overview}
\end{sidewaysfigure}

Gao--Lam--Xu~\cite[Conjecture~1.4]{gao2025electrical} conjectured the existence of an \dfn{electrical canonical basis} of the grove algebra in all degrees. As we explain below, the existence of this electrical canonical basis would yield a proof of \cref{conj:main} for all $m \geq 1$ when $P=\delta_n$.

The rest of the paper is structured as follows. In \cref{sec:comb_background}, we go over combinatorial background on rowmotion/promotion of staircase plane partitions, a.k.a., fans of Dyck paths. In~\cref{sec:crystals}, we discuss the crystal perspective, and explain how it yields that promotion of height two plane partitions is in equivariant bijection with rotation of $3$-noncrossing perfect matchings. In~\cref{sec:electrical}, we discuss the space of electrical networks and its coordinate ring, the grove algebra. We use the existence of an electrical canonical basis in degree two, the bush basis, to complete the proof of \cref{thm:main} in this section. Finally, in \cref{sec:future}, we discuss future directions.

\begin{remark} 
We know a different way to prove \cref{thm:main} that avoids electrical networks altogether and just uses crystals. Namely, Westbury~\cite{westbury2016invariant} proved a general CSP result for invariant tensors of crystals. Moreover, in the case of interest to us where the action is rotation of $3$-noncrossing perfect matchings, Rubey and Westbury~\cite[Theorem 3.5]{rubey2015combinatorics} showed that the sieving polynomial in question is a certain explicit sum of principal specializations of Schur polynomials. (They used work of Sundaram~\cite{sundaram1986combinatorics} to obtain this sum formula.) A result of Krattenthaler~\cite[Theorem 2]{krattenthaler1995major} then shows that this sum is the same as the product formula appearing in \cref{conj:main}. However, we prefer the approach we take in this paper, using the space of electrical networks, for a few reasons. First, the appearance of the product formula in \cref{conj:main} is more transparent in our approach. Second, we believe our approach should generalize. Indeed, we already mentioned how the existence of an electrical canonical basis would yield the CSP for higher height staircase plane partitions. Even more generally, the paradigm of realizing a combinatorial action of interest via a nice basis of a coordinate ring of a homogeneous variety might work to address other cases of \cref{conj:main}.
\end{remark}

\begin{remark}
On the other hand, we also believe there is a way to prove \cref{thm:main} that avoids the use of crystals. Namely, there is an explicit bijection between staircase plane partitions of height two and $3$-noncrossing perfect matchings due to Chen, Deng, Du, Stanley, and Yan~\cite{chen2007crossings}. We believe that it can be shown in a purely combinatorial manner that this bijection sends promotion to rotation. However, we prefer to include a discussion of crystals because they provide another interesting algebraic avatar of promotion. Moreover, we wish to highlight the mysterious ``duality'' between invariant tensors in powers of spin representations and the coordinate ring of the Lagrangian Grassmannian.
\end{remark}

\subsection*{Acknowledgments}

The first author thanks Zach Hamaker for putting him in touch with the second author. The first author was partially supported by a Simons Foundation travel grant. The third author was partially supported by Discovery Grants RGPIN-2021-02391 and RGPIN-2021-02568 from NSERC.

\section{Combinatorial background} \label{sec:comb_background}

\subsection{Staircase plane partitions, rowmotion and promotion}

Fix positive integers, $n, m \geq 1$. A \dfn{staircase plane partition} of size $n$ and height $m$ is a triangular array of nonnegative integers
\[
\pi = (\pi_{i,j}) \quad \text{for } 1 \leq i,j \leq n \text{ and } i+j\le n+1
\]
such that: $\pi_{i,j} \geq \pi_{i+1,j}$ for all $(i,j)$; $\pi_{i,j} \geq \pi_{i,j+1}$ for all $(i,j)$; and $\pi_{1,1} \leq m$. As in the introduction, we denote the set of staircase plane partitions of size $n$ and height $m$ by~$\mathcal{PP}^m(\delta_n)$. We view an element of $\mathcal{PP}^m(\delta_n)$ as a triangular array of left-justified boxes that has $n$ boxes in the first row, $n-1$ boxes in the second row, and so on down to one box in the $n$th row, filled with nonnegative integers that weakly decrease along rows and down columns. We use matrix notation, so position~$(i,j)$ means the box in the $i$th row and $j$th column from the top left.

The \dfn{toggle} at position $(i,j)$ is a piecewise-linear involution on the set of staircase plane partitions $\tau_{i,j}\colon \mathcal{PP}^m(\delta_n)\to\mathcal{PP}^m(\delta_n)$ defined by:
\[
\tau_{i,j}(\pi)_{k,l} :=
\begin{cases}
\pi_{k,l}, & \text{if } (k,l) \neq (i,j), \\
\min(\pi_{i,j-1}, \pi_{i-1,j}) + \max(\pi_{i+1,j}, \pi_{i,j+1}) - \pi_{i,j}, & \text{if } (k,l) = (i,j),
\end{cases}
\]
with the boundary conditions
\[
\pi_{0,j} := m, \quad \pi_{i,0} := m, \quad \pi_{i,j} := 0 \text{ if $i+j>n+1$}.
\]
For $-n+1 \leq k \leq n-1$, let $F_k$ be the composition of toggles along the $k$th diagonal:
\[
F_k := \prod_{\substack{1 \leq i \leq n \\ 1 \leq j \leq n \\ j - i = k}} \tau_{i,j}.
\]
For $1 \leq k \leq n$, let $R_k$ be the composition of toggles along the $k$th antidiagonal:
\[
R_k := \prod_{\substack{1 \leq i \leq n \\ 1 \leq j \leq n \\ i + j - 1 = k}} \tau_{i,j}.
\]
\Cref{fig:toggle-composites} illustrates these two families of commuting toggle composites.

\begin{figure}
\centering
\begin{tikzpicture}[
  scale=.78,
  celltext/.style={font=\scriptsize},
  title/.style={font=\small},
]
\begin{scope}
\node[title] at (2.5,.85) {diagonal composites $F_k$};
\foreach \i in {1,...,5} {
  \pgfmathtruncatemacro{\last}{6-\i}
  \foreach \j in {1,...,\last} {
    \pgfmathtruncatemacro{\fk}{\j-\i}
    \pgfmathtruncatemacro{\shade}{18+8*(\fk+4)}
    \path[draw,fill=cyan!\shade!white] (\j-1,-\i) rectangle (\j,-\i+1);
    \node[celltext] at (\j-.5,-\i+.5) {$F_{\fk}$};
  }
}
\end{scope}

\begin{scope}[xshift=7.1cm]
\node[title] at (2.5,.85) {antidiagonal composites $R_k$};
\foreach \i in {1,...,5} {
  \pgfmathtruncatemacro{\last}{6-\i}
  \foreach \j in {1,...,\last} {
    \pgfmathtruncatemacro{\rk}{\i+\j-1}
    \pgfmathtruncatemacro{\shade}{18+12*(\rk-1)}
    \path[draw,fill=orange!\shade!white] (\j-1,-\i) rectangle (\j,-\i+1);
    \node[celltext] at (\j-.5,-\i+.5) {$R_{\rk}$};
  }
}
\end{scope}
\end{tikzpicture}
\caption{The toggle composites $F_k$ and $R_k$ on a staircase of size~$5$. Cells with the same label are toggled together in one composite; within each composite no two toggled cells are horizontal or vertical neighbors, so the corresponding local toggles commute.}
\label{fig:toggle-composites}
\end{figure}
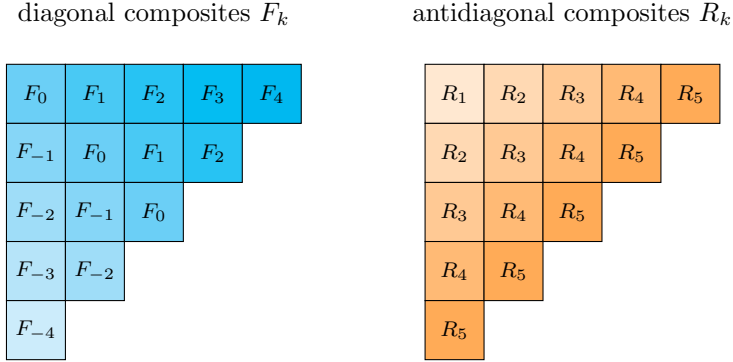

We define \dfn{rowmotion} $\mathrm{Row}\colon \mathcal{PP}^m(\delta_n) \to \mathcal{PP}^m(\delta_n)$ as
\begin{equation}\label{eq:rowmotion-definition}
\mathrm{Row} := R_n \cdot R_{n-1} \cdot \dots \cdot R_1
\end{equation}
and \dfn{promotion} $\mathrm{Pr}\colon \mathcal{PP}^m(\delta_n) \to \mathcal{PP}^m(\delta_n)$ as
\begin{equation}\label{eq:promotion-definition}
\mathrm{Pr} := F_{n-1} \cdot F_{n-2} \cdot \dots \cdot F_{-n+2} \cdot F_{-n+1}.
\end{equation}

\begin{thm}[{Striker--Williams~\cite{striker2012promotion}; see also~\cite{hopkins2020cyclic}}] \label{thm:striker_williams}
$\mathrm{Row}\colon \mathcal{PP}^m(\delta_n)\to\mathcal{PP}^m(\delta_n)$ and $\mathrm{Pr}\colon \mathcal{PP}^m(\delta_n)\to\mathcal{PP}^m(\delta_n)$ are conjugate to one another in the group generated by the toggles.
\end{thm}

\subsection{Fans of Dyck paths}

We now describe a different way of viewing promotion of staircase plane partitions.

A \dfn{Dyck path} of semilength $n$ (and length $2n$) is a lattice walk from $(0,0)$ to $(n,n)$ using only unit steps in north and east direction that stays above the $x=y$ line. We represent a Dyck path of semilength $n$ as a sequence $(0=d_0,d_1,\dots,d_{2n}=0)$ of heights, where $d_i$ is the distance between the path and the $x=y$ line in $x$-direction after the $i$th step. We say that a Dyck path $D = (d_0,d_1, \dots, d_{2n})$ is nested by a Dyck path
$D' = (d'_0,d'_1, \dots, d'_{2n})$ if for all $0\le i\le 2n$ we have $d_i \le d_i'$ and we write~$D\le D'$.

An \dfn{$m$-fan of Dyck paths} is an $m$-tuple $(D_m,D_{m-1},\dots,D_1)$ of Dyck paths of the same semilength such that $D_1 \le D_2 \le \dots \le D_m$.
We denote with $\mathcal{D}^m_n$ the set of $m$-fans of Dyck paths of semilength $n$. We represent an $m$-fan $\mathcal{F}$ of Dyck paths with a sequence of vectors of heights $(\emptyset=\mu^0, \mu^1, \dots, \mu^{2n}=\emptyset)$, where for each $\mu^i = (\mu^i_m,\mu^i_{m-1},\dots, \mu_1^i) \in \mathbb{Z}_{\ge 0}^m$ the entries are weakly decreasing and the sequence $(\mu^0_k,\mu^1_k,\dots,\mu^{2n}_k)$ gives the Dyck path with index $k$ in $\mathcal{F}$.
For example, the fan of Dyck paths on the left in \cref{fig:bijection PP-fans} is given by
\[
(0000,\! 1111,\! 2222, 3331, 4442, 5331, 6220, 5311, 6420, 5311, 4200, 3111, 2220,\! 1111,\! 0000)
\]

\begin{figure}
\begin{center}
\begin{tikzpicture}
\node at (0,0) {\begin{tikzpicture}[scale=0.8, every node/.style={font=\small}]
\def\pp{
{4,3,3,3,3,1},
{3,3,2,2,2},
{3,3,2,1},
{1,1,1},
{1,0},
{0}
}
\def\n{7}
\def\offset{0.05}
\foreach \row in {1,...,\n} {
	\draw (0,\row) rectangle (\row,\row);
}
\foreach \col in {1,...,\n} {
	\draw (\col-1,\n) rectangle (\col-1,\col-1);
}
\foreach \row [count=\r from 0] in \pp {
    \foreach \val [count=\c from 0] in \row {
        \node at (\c+0.5,\n-\r-0.5) {\val};
    }
}
\draw[thick,dashed] (0,0) -- (\n,\n);
\draw[red,thick]
(0,0) -- (0,2) -- (1,2) -- (1,3) -- (3,3) -- (3,4) -- (4,4) -- (4,5) -- (5,5) -- (5,6) -- (6,6) -- (6,7) -- (7,7);
\draw[blue,thick]
({0-\offset},{0+\offset}) --
({0-\offset},{4+\offset}) --
({3-\offset},{4+\offset}) --
({3-\offset},{5+\offset}) --
({5-\offset},{5+\offset}) --
({5-\offset},{7+\offset}) --
({7-\offset},{7+\offset});
\draw[green,thick]
({0-2*\offset},{0+2*\offset}) --
({0-2*\offset},{4+2*\offset}) --
({2-2*\offset},{4+2*\offset}) --
({2-2*\offset},{6+2*\offset}) --
({5-2*\offset},{6+2*\offset}) --
({5-2*\offset},{7+2*\offset}) --
({7-2*\offset},{7+2*\offset});
\draw[purple,thick]
({0-3*\offset},{0+3*\offset}) --
({0-3*\offset},{6+3*\offset}) --
({1-3*\offset},{6+3*\offset}) --
({1-3*\offset},{7+3*\offset}) --
({7-3*\offset},{7+3*\offset});
\end{tikzpicture}};    
\node at (3.5,0) {\Large $\xrightarrow{\mathrm{Pr}}$};
\node at (7.5,0) {\begin{tikzpicture}[scale=0.8, every node/.style={font=\small}]
\def\pp{
{4,4,4,4,2,1},
{4,3,3,3,1},
{3,3,1,0},
{3,3,1},
{1,1},
{1}
}
\def\n{7}
\def\offset{0.05}
\foreach \row in {1,...,\n} {
	\draw (0,\row) rectangle (\row,\row);
}
\foreach \col in {1,...,\n} {
	\draw (\col-1,\n) rectangle (\col-1,\col-1);
}
\foreach \row [count=\r from 0] in \pp {
    \foreach \val [count=\c from 0] in \row {
        \node at (\c+0.5,\n-\r-0.5) {\val};
    }
}
\draw[thick,dashed] (0,0) -- (\n,\n);
\draw[red,thick]
(0,0) -- (0,1) -- (1,1) -- (1,2) -- (2,2) -- (2,3) -- (3,3) -- (3,5) -- (4,5) -- (5,5) -- (5,6) -- (6,6) -- (6,7) -- (7,7);
\draw[blue,thick]
({0-\offset},{0+\offset}) --
({0-\offset},{3+\offset}) --
({2-\offset},{3+\offset}) --
({2-\offset},{5+\offset}) --
({4-\offset},{5+\offset}) --
({4-\offset},{6+\offset}) --
({5-\offset},{6+\offset}) --
({5-\offset},{7+\offset}) --
({7-\offset},{7+\offset});
\draw[green,thick]
({0-2*\offset},{0+2*\offset}) --
({0-2*\offset},{3+2*\offset}) --
({2-2*\offset},{3+2*\offset}) --
({2-2*\offset},{5+2*\offset}) --
({4-2*\offset},{5+2*\offset}) --
({4-2*\offset},{7+2*\offset}) --
({7-2*\offset},{7+2*\offset});
\draw[purple,thick]
({0-3*\offset},{0+3*\offset}) --
({0-3*\offset},{5+3*\offset}) --
({1-3*\offset},{5+3*\offset}) --
({1-3*\offset},{6+3*\offset}) --
({4-3*\offset},{6+3*\offset}) --
({4-3*\offset},{7+3*\offset}) --
({7-3*\offset},{7+3*\offset});
\end{tikzpicture}};
\end{tikzpicture}
\end{center}
\caption{An example of the bijection $\Phi$ between staircase plane partitions and fans of Dyck paths. We also depict how promotion behaves on these objects.}
\label{fig:bijection PP-fans}
\end{figure}

Staircase plane partitions of size $n-1$ and height $m$ are naturally in bijection with $m$-fans of Dyck paths of semilength $n$, as we now explain. A Dyck path $D$ of length $2n$ is uniquely determined by the set $\Lambda(D)$ of unit squares that lie between the path and lines $x=0$ and $y=n$. For a plane partition $\pi\in \mathcal{PP}^m(\delta_{n-1})$ let~$\pi^{\ge i}$ be the set of all boxes whose entry is at least $i$. Then for all $1\le i\le m$ there is a unique Dyck path $D_i$ with $\Lambda(D_i) = \pi^{\ge i}$ and $(D_m,D_{m-1},\dots,D_1) \in \mathcal{D}^m_n$. See \cref{fig:bijection PP-fans} for an example. We denote this map by $\Phi\colon \mathcal{PP}^m(\delta_{n-1}) \to \mathcal{D}^m_n$.

We now define promotion on fans of Dyck paths. For a vector $\rho \in \mathbb{Z}^m$ its \dfn{dominant representative} $\mathrm{dom}(\rho)$ is the vector obtained from $\rho$ by sorting the absolute values of its entries in weakly decreasing order. Let $\mathcal{F} = (\emptyset=\mu^0,\mu^1,\mu^2, \ldots, \mu^{2n}=\emptyset)$ be an $m$-fan of Dyck paths of semilength $n$. Then for each $i=1,\ldots,2n-1$, we define $\mathrm{BK}_i(\mathcal{F}) := (\emptyset=\mu^0, \dots, \mu^{i-1},\lambda^{i},\mu^{i+1}, \dots, \mu^{2n}=\emptyset)$, where
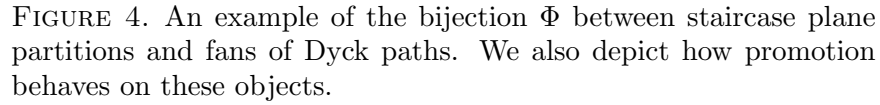
    \[
    \lambda^{i} := \mathrm{dom}(\mu^{i-1}+\mu^{i+1}-\mu^{i}).
    \]
We call these maps $\mathrm{BK}_i$ because they are similar to \dfn{Bender--Knuth involutions} acting on tableaux. Indeed, note that each $\mathrm{BK}_i\colon \mathcal{D}^m_n \to \mathcal{D}^m_n$ is a well-defined involution. Finally, we define the \dfn{promotion} $\mathrm{Pr}\colon\mathcal{D}^m_n\to\mathcal{D}^m_n$ of $m$-fans of Dyck paths as the composition
    \[
    \mathrm{Pr} := \mathrm{BK}_{2n-1}\cdot \mathrm{BK}_{2n-2}\cdot \ldots \cdot \mathrm{BK}_{1}.
    \]
(Note however that $\mathrm{BK}_1$ and $\mathrm{BK}_{2n-1}$ always act as the identity, so this is the same as defining $\mathrm{Pr} := \mathrm{BK}_{2n-2}\cdot \mathrm{BK}_{2n-3}\cdot \ldots \cdot \mathrm{BK}_{2}$.)

These maps are well-defined involutions on $\mathcal{D}^m_n$.  For a single path the rule is transparent: at position $i$, a valley becomes a peak, a peak of height greater than~$1$ becomes a valley, a peak of height $1$ is fixed, and a double rise or double fall is fixed.  For a fan, the same rule is applied to the multiset of local segments, followed by sorting to restore the nesting order.

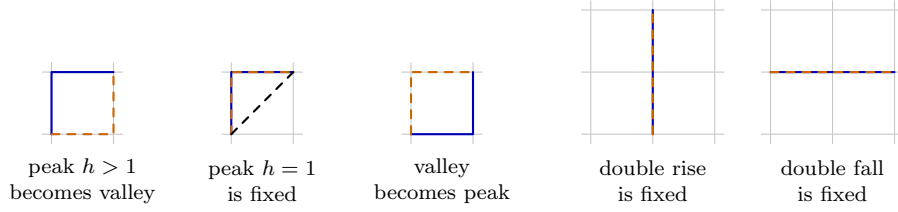
\begin{figure}
\centering
\begin{tikzpicture}[scale=0.82, line cap=round, line join=round,
  every node/.style={font=\scriptsize}]
  \newcommand{\panel}[4]{
    \begin{scope}[xshift=#1]
      \draw[gray!45] (-.15,0) -- (2.15,0);
      \draw[gray!45] (-.15,1) -- (2.15,1);
      \draw[gray!45] (-.15,2) -- (2.15,2);
      \draw[gray!45] (0,-.15) -- (0,2.15);
      \draw[gray!45] (1,-.15) -- (1,2.15);
      \draw[gray!45] (2,-.15) -- (2,2.15);
      #3
      \node[align=center] at (1,-.75) {#2};
      \node[align=center,font=\scriptsize] at (1,2.35) {#4};
    \end{scope}}

\newcommand{\panelsmall}[4]{
    \begin{scope}[xshift=#1]
      \draw[gray!45] (-.15,0) -- (1.15,0);
      \draw[gray!45] (-.15,1) -- (1.15,1);
      \draw[gray!45] (0,-.15) -- (0,1.15);
      \draw[gray!45] (1,-.15) -- (1,1.15);
      #3
      \node[align=center] at (0.5,-.75) {#2};
      \node[align=center,font=\scriptsize] at (0.5,1.35) {#4};
    \end{scope}}

  \panelsmall{0cm}{peak $h>1$\\becomes valley}{
    \draw[blue!70!black,thick] (0,0) -- (0,1) -- (1,1);
    \draw[orange!80!black,thick,dashed] (0,0) -- (1,0) -- (1,1);
  }{}
  \panelsmall{2.9cm}{peak $h=1$\\is fixed}{
    \draw[blue!70!black,thick] (0,0) -- (0,1) -- (1,1);
    \draw[orange!80!black,thick,dashed] (0,0) -- (0,1) -- (1,1);
    \draw[thick,dashed] (0,0) -- (1,1);
  }{}
  \panelsmall{5.8cm}{valley\\becomes peak}{
    \draw[blue!70!black,thick] (0,0) -- (1,0) -- (1,1);
    \draw[orange!80!black,thick,dashed] (0,0) -- (0,1) -- (1,1);
  }{}
  \panel{8.7cm}{double rise\\is fixed}{
    \draw[blue!70!black,thick] (1,0) -- (1,1) -- (1,2);
    \draw[orange!80!black,thick,dashed] (1,0) -- (1,1) -- (1,2);
  }{}
  \panel{11.6cm}{double fall\\is fixed}{
    \draw[blue!70!black,thick] (0,1) -- (1,1) -- (2,1);
    \draw[orange!80!black,thick,dashed] (0,1) -- (1,1) -- (2,1);
  }{}
\end{tikzpicture}
\caption{The local effect of $\BK_i$ on one path.  Solid blue shows the local shape before applying $\BK_i$, and dashed orange shows the local shape after applying it.  In a fan, the resulting segments are sorted by height.}
\label{fig:bk-local}
\end{figure}

One way to compute promotion of a fan $\mathcal{F}$ is in terms of the following diagram:
\begin{equation}\label{eq:promotion-diag}
  \begin{tikzcd}
    \mu^{1} \rar
    & \mu^{2} \rar
    & \cdots \rar
    & \emptyset \\
    \emptyset \rar \uar
        \ar[->,blue,rounded corners,to path={
           ([yshift=0.0em,xshift=-1.8em]\tikzcdmatrixname-2-1.center)
        -- ([yshift=2.0em,xshift=-1.8em]\tikzcdmatrixname-1-1.center) node[above,xshift=8em]{$\mathcal{F}$}
        -- ([yshift=2.0em,xshift=0.0em]\tikzcdmatrixname-1-4.center)
        }]{}
        \ar[->,orange,rounded corners,to path={
           ([yshift=-1.1em,xshift=-1.8em]\tikzcdmatrixname-2-1.center)
        -- ([yshift=-1.1em,xshift=-1.8em]\tikzcdmatrixname-2-2.center)
        -- ([yshift=1.5em,xshift=-1.8em]\tikzcdmatrixname-1-2.center) node[below right,xshift=2em]{$\mathrm{BK}_1(\mathcal{F})$}
        -- ([yshift=1.5em,xshift=0.0em]\tikzcdmatrixname-1-4.center)
        }]{}
        \ar[->,red,rounded corners,to path={
           ([yshift=-2.0em,xshift=-1.8em]\tikzcdmatrixname-2-1.center) node[below right, xshift=1em]{$\mathrm{BK}_{2n-1} \circ \cdots \circ \mathrm{BK}_1(\mathcal{F})$}
        -- ([yshift=-2.0em,xshift=-1.8em]\tikzcdmatrixname-2-4.center)
        -- ([yshift=1.0em,xshift=-1.8em]\tikzcdmatrixname-1-4.center)
        -- ([yshift=1.0em,xshift=0.0em]\tikzcdmatrixname-1-4.center)
        }]{}
    & \lambda^{1} \rar \uar
    & \cdots \rar \uar
    & \lambda^{2n-1} \uar
\end{tikzcd}
\end{equation}
We write $\mathcal{F}=(\emptyset=\mu^0, \mu^1,\dots, \mu^{2n}=\emptyset)$ on the corners on the left and top border of the diagram and we apply the local rule $\lambda = \mathrm{dom}(\kappa+\nu -\mu)$ to any square~\begin{tikzcd}
  \mu \rar
    & \nu \\
  \kappa \uar \rar
    & \lambda \uar
  \end{tikzcd}.
The promotion $\Pr(\mathcal{F}) = (\emptyset=\lambda^0, \lambda^1,\dots, \lambda^{2n}=\emptyset)$ can be read off at the bottom and right border of the diagram.

A key observation about the map $\Phi\colon \mathcal{PP}^m(\delta_{n-1}) \to \mathcal{D}^m_n$ is that it intertwines toggles and Bender--Knuth involutions and is therefore promotion-equivariant. More formally:

\begin{lemma} \label{lem:pp_dyck_bij}
    For $\pi\in \mathcal{PP}^m(\delta_{n-1})$ we have
\begin{equation}
    \label{eq:toggle-bk-intertwine}
    \Phi(F_{i-n}(\pi)) = \mathrm{BK}_{i}(\Phi(\pi))
\end{equation}
     for all $2\le i\le 2n-2$, and hence
     \[
     \Phi(\mathrm{Pr}(\pi)) = \mathrm{Pr}(\Phi(\pi)).\]
\end{lemma}
\begin{proof}
Fix $i$ with $2\leq i\leq 2n-2$. The cells $(p,q)$ on the diagonal
\[
q-p=i-n
\]
are exactly those that can affect the height vector $\mu^i$ of the corresponding fan. Toggling these cells leaves every $\mu^r$ with $r\neq i$ unchanged. Moreover, no two cells on this diagonal are related by a cover, so the corresponding toggles commute. It is therefore enough to examine the effect of toggling a single cell.

Suppose that the entry in this cell is $x$. We first consider the configuration shown in \cref{fig:toggle-local}. Write the entries in the north, west, south, and east neighboring cells as
\[
N=x+a+d,\qquad W=x+a,\qquad S=x-b,\qquad E=x-b-c,
\]
where $a,b,c,d\geq0$. In particular,
\[
\min(N,W)=x+a,
\qquad
\max(S,E)=x-b.
\]

Recall that, for each $t$, one path in $\Phi(\pi)$ is the boundary path associated with the cells of $\pi$ whose entries are at least $t$. The local segments of these paths that meet the cell containing $x$ fall into four families.

For the $a$ values of $t$ satisfying $ x<t\leq x+a, $
the north and west neighbors have entries at least $t$, whereas the cell containing $x$ does not. The corresponding path follows the left side and then the upper side of the cell. Hence there are $a$ local $NE$ peaks.

Similarly, for the $b$ values of $t$ satisfying $x-b<t\leq x,$
the corresponding path follows the lower side and then the right side of the cell. Hence there are $b$ local $EN$ valleys.

The differences between the two upper neighbors and between the two lower neighbors account for the remaining local segments. For the $d$ values of $t$ satisfying $x+a<t\leq x+a+d,$
the path has a straight segment along the upper side of the cell.
Similarly, for the $c$ values of $t$ satisfying $x-b-c<t\leq x-b,$
the path has a straight segment along the right side of the cell.
These segments are determined entirely by the neighboring entries
and are therefore unaffected by the toggle. All other paths do not touch the cell.

The toggle replaces $x$ by
\[
x'=\min(N,W)+\max(S,E)-x=x+a-b.
\]
Expressing the neighboring entries in terms of $x'$ gives
\[
N=x'+b+d,\qquad W=x'+b,\qquad
S=x'-a,\qquad E=x'-a-c.
\]
The same count therefore shows that, after the toggle, there are $b$
local $NE$ peaks and $a$ local $EN$ valleys. The $c$ and $d$ straight
segments remain unchanged. Thus the toggle exchanges the multiplicities of the peak and valley segments and fixes every other local segment.

The figure shows only one of the four possible placements of the $c$- and $d$-segments.
If $W>N$, the $d$-segments lie along the left side rather than the upper side; if $E>S$, the $c$-segments lie along the lower side rather than the right side. In every case, the same calculation exchanges the numbers of peaks and valleys and leaves the straight segments fixed.

This is precisely the one-path local rule for $\BK_i$. The only exception occurs for an $NE$ peak of height $1$: replacing it by a valley would take the path below the diagonal, and the absolute value in the definition of $\dom$ reflects it back, so that this segment is fixed. For a fan, the resulting local segments are sorted by height, which is exactly the sorting in
\[
\mu^i\longmapsto
\dom\bigl(\mu^{i-1}+\mu^{i+1}-\mu^i\bigr).
\]
Toggling all cells on the diagonal therefore proves \eqref{eq:toggle-bk-intertwine}. Composing these identities for $2\leq i\leq 2n-2$, and including the trivial factors $\BK_1$ and $\BK_{2n-1}$, proves the promotion statement.
\end{proof}

\begin{figure}[t]
\centering
\begin{tikzpicture}[
  x=1.6cm,
  y=1.6cm,
  line cap=round,
  line join=round,
  entry/.style={
    font=\footnotesize,
    fill=white,
    inner sep=1.4pt
  },
  segmentlabel/.style={
    font=\scriptsize,
    fill=white,
    inner sep=.6pt
  },
  peak/.style={
    draw=orange!85!black,
    line width=1.6pt,
    postaction={decorate},
    decoration={
      markings,
      mark=at position .72 with
        {\arrow{Stealth[length=1.8mm]}}
    }
  },
  valley/.style={
    draw=blue!70!black,
    line width=1.6pt,
    postaction={decorate},
    decoration={
      markings,
      mark=at position .72 with
        {\arrow{Stealth[length=1.8mm]}}
    }
  },
  fixed/.style={
    draw=gray!75!black,
    dashed,
    line width=1.15pt,
    postaction={decorate},
    decoration={
      markings,
      mark=at position .62 with
        {\arrow{Stealth[length=1.6mm]}}
    }
  },
  grid/.style={
    draw=gray!55,
    line width=.55pt
  }
]
  \newcommand{\crossgrid}{
    \draw[grid] (0,0) rectangle (1,1);
    \draw[grid] (0,1) rectangle (1,2);
    \draw[grid] (-1,0) rectangle (0,1);
    \draw[grid] (1,0) rectangle (2,1);
    \draw[grid] (0,-1) rectangle (1,0);
  }

  \begin{scope}
    \crossgrid

    \node[entry] at (.5,.5) {$x$};
    \node[entry] at (.5,1.5) {$x+a+d$};
    \node[entry] at (-.5,.5) {$x+a$};
    \node[entry] at (.5,-.5) {$x-b$};
    \node[entry] at (1.5,.5) {$x-b-c$};

    \draw[peak]
      (0,0) -- (0,1) -- (1,1);
    \node[segmentlabel,text=orange!85!black]
      at (.24,.78) {$a$};

    \draw[valley]
      (0,0) -- (1,0) -- (1,1);
    \node[segmentlabel,text=blue!70!black]
      at (.76,.22) {$b$};

    \draw[fixed]
      (-.06,1.06) -- (1.06,1.06);
    \node[
      segmentlabel,
      text=gray!75!black,
      anchor=south west
    ] at (-.06,1.07) {$d$};

    \draw[fixed]
      (1.06,-.06) -- (1.06,1.06);
    \node[
      segmentlabel,
      text=gray!75!black,
      anchor=north west
    ] at (1.07,-.06) {$c$};
  \end{scope}

  \node[align=center,font=\scriptsize] at (3.08,.55)
    {$x'=x+a-b$\\[-1pt]
     \Large $\xrightarrow{\ \tau\ }$};

  \begin{scope}[xshift=8.3cm]
    \crossgrid

    \node[entry] at (.5,.5) {$x'$};
    \node[entry] at (.5,1.5) {$x'+b+d$};
    \node[entry] at (-.5,.5) {$x'+b$};
    \node[entry] at (.5,-.5) {$x'-a$};
    \node[entry] at (1.5,.5) {$x'-a-c$};

    \draw[peak]
      (0,0) -- (0,1) -- (1,1);
    \node[segmentlabel,text=orange!85!black]
      at (.24,.78) {$b$};

    \draw[valley]
      (0,0) -- (1,0) -- (1,1);
    \node[segmentlabel,text=blue!70!black]
      at (.76,.22) {$a$};

    \draw[fixed]
      (-.06,1.06) -- (1.06,1.06);
    \node[
      segmentlabel,
      text=gray!75!black,
      anchor=south west
    ] at (-.06,1.07) {$d$};

    \draw[fixed]
      (1.06,-.06) -- (1.06,1.06);
    \node[
      segmentlabel,
      text=gray!75!black,
      anchor=north west
    ] at (1.07,-.06) {$c$};
  \end{scope}
\end{tikzpicture}

\caption{The local path segments in the case $N=x+a+d$, $W=x+a$, $S=x-b$, and $E=x-b-c$. Orange denotes the peak segments, blue denotes the valley segments, and dashed gray denotes the straight segments fixed by the toggle. The multiplicities $a$ and $b$ are exchanged, while $c$ and $d$ remain unchanged.}
\label{fig:toggle-local}
\end{figure}
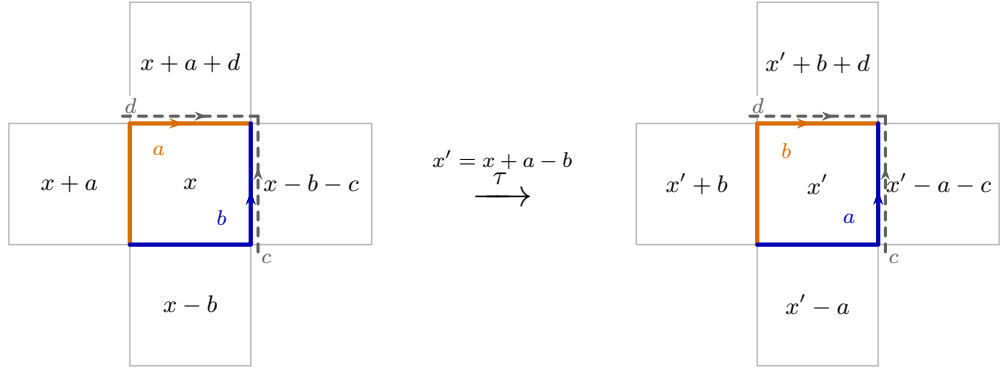

\section{Crystal perspective} \label{sec:crystals}

\dfn{Crystal bases}, introduced by Kashiwara and Lusztig, provide a combinatorial skeleton of representations of quantum groups. For background on crystals, consult~\cite{bump2017crystal}; we only briefly review the theory here. A crystal is a set $B$ equipped with:
\[
e_i, f_i : B \to B \cup \{\emptyset\}, \quad \varepsilon_i, \varphi_i : B \to \mathbb{Z}, \quad \mathrm{wt}: B \to P,
\]
where $P$ is the weight lattice. The operators $e_i$ and $f_i$ are \dfn{raising and lowering operators}, and~$\mathrm{wt}$ gives the weight of an element.
Given two crystals $B$ and $C$, there is a combinatorial construction for the tensor product $B\otimes C$ which corresponds to the crystal of the tensor product of the corresponding representations. The elements in $B\otimes C$ are of the form~$b\otimes c$ where $b\in B$ and $c\in C$. The data of a crystal can be represented by an edge-colored directed graph, called \dfn{crystal graph}, where we have an edge colored $i$ from $a$ to~$b$ if $f_i(a) = b$.
The connected components of a crystal graph correspond to irreducible representations and the number of vertices in a connected component is its dimension.

Each connected component has a unique source, which we call the \dfn{highest weight element} of its component.
A special role are isolated vertices in a crystal graph, which are the \dfn{highest weight elements of weight zero}.

Abstractly, \dfn{promotion} can be defined as an action on highest weight elements of weight zero in a tensor product of crystals $B^{\otimes n}$:
\begin{equation}\label{eq:crystal commutor}
    \Pr(u) := \sigma_{B,B^{\otimes (n-1)}}(u),
\end{equation}
where $\sigma$ is the \dfn{crystal commutor} introduced by Henriques and Kamnitzer~\cite{henriques2006crystals} and built from Lusztig's involution.

\subsection{Spin invariants and fans of Dyck paths}

For type $B_r$, the \dfn{spin crystal} $B_r^{\mathrm{spin}}$ consists of $r$-tuples $(\pm,\dots,\pm)$ with crystal operators acting by flipping adjacent signs. See for example \Cref{fig:spin crystal B3}.

\begin{figure}
\begin{center}
\newcommand{\threecell}[3]{
  \begin{tikzpicture}[scale=0.24]
    \draw[line width=0.7pt] (0,0) rectangle (1,3);
    \draw[line width=0.7pt] (0,1) -- (1,1);
    \draw[line width=0.7pt] (0,2) -- (1,2);
    \node at (0.5,2.5) {\tiny $#1$};
    \node at (0.5,1.5) {\tiny $#2$};
    \node at (0.5,0.5) {\tiny $#3$};
  \end{tikzpicture}
}
\begin{tikzpicture}[scale=0.25]
\node (A) at (0,0) {\threecell{+}{+}{+}};
\node (B) at (7,0) {\threecell{-}{+}{+}};
\node (C) at (14,0) {\threecell{+}{-}{+}};
\node (D1) at (21,1.6) {\threecell{-}{-}{+}};
\node (D2) at (21,-1.6) {\threecell{+}{+}{-}};
\node (E) at (28,0) {\threecell{-}{+}{-}};
\node (F) at (35,0) {\threecell{+}{-}{-}};
\node (G) at (42,0) {\threecell{-}{-}{-}};
\draw[->, draw=green!60!black,ultra thick]  (A) -- node[above] {\small 3} (B);
\draw[->, draw=red!60!black,ultra thick]  (B) -- node[above] {\small 2} (C);
\draw[->, draw=green!60!black,ultra thick]  (C) -- node[above] {\small 3} (D1);
\draw[->, draw=blue!60!black,ultra thick]  (C) -- node[below] {\small 1} (D2);
\draw[->, draw=blue!60!black,ultra thick]  (D1) -- node[above] {\small 1} (E);
\draw[->, draw=green!60!black,ultra thick]  (D2) -- node[below] {\small 3} (E);
\draw[->, draw=red!60!black,ultra thick]  (E) -- node[above] {\small 2} (F);
\draw[->, draw=green!60!black,ultra thick]  (F) -- node[above] {\small 3} (G);
\end{tikzpicture}
\end{center}
\caption{The spin crystal $B^{\mathrm{spin}}_3$ for type $B_3$.}
\label{fig:spin crystal B3}
\end{figure}

For an $r$-fan of Dyck paths $\mathcal{F}=(\emptyset=\mu^0,\mu^1,\dots,\mu^{2n}=\emptyset)$, each consecutive difference $\mu^i-\mu^{i-1}$ is a vector of the form $(\pm 1,\pm 1,\dots,\pm 1)$, which we naturally identify with an element of $B_r^{\mathrm{spin}}$. But note that the Dyck path fan height vectors are written as $\mu^i=(\mu_r^i,\mu_{r-1}^i,\dots,\mu_1^i)$, while the sign columns in the spin-crystal figures, like \cref{fig:spin crystal B3}, are read from top to bottom in the reverse order as $(\epsilon_1,\epsilon_2,\dots,\epsilon_r)$. With this identification, we view an $r$-fan of semilength $n$ as an element of $(B_r^{\mathrm{spin}})^{\otimes 2n}$. As observed by Pappe et al.~\cite{pappe2024promotion}, the $r$-fans of Dyck paths precisely correspond to the highest weight elements of weight zero in this crystal, and equation~\eqref{eq:promotion-diag} is exactly Lenart's realization of the crystal commutor in terms of local rules. So:

\begin{lemma}[{\cite{pappe2024promotion}}]\label{lem:promotion_crystal}
Highest weight elements of weight zero in the crystal $(B_r^{\mathrm{spin}})^{\otimes 2n}$ can be naturally identified with $r$-fans of Dyck paths of semilength $n$. Under this identification, promotion in terms of the crystal commutor in \eqref{eq:crystal commutor} corresponds to promotion on $r$-fans of Dyck paths.
\end{lemma}

\begin{remark} 
There are two conventions for the order of tensor factors in crystal graphs. Here for convenience we are using the Kashiwara conventions, while~\cite{pappe2024promotion} use the conventions from~\cite{bump2017crystal}.
\end{remark}

\subsection{Symplectic invariants and noncrossing perfect matchings}

In order to prove our cyclic sieving result, we need a bijection from $2$-fans of Dyck paths to $3$-noncrossing perfect matchings which intertwines promotion and rotation. (This is because $3$-noncrossing perfect matchings index a basis of the degree two part of the grove algebra.)

Let $n \ge 1$ and $r \ge 1$. A \dfn{perfect matching} on $[2n] := \{1,2,\dots,2n\}$ is a partition of~$[2n]$ into $n$ disjoint pairs, which we represent as chords in a chord diagram with $2n$ vertices. A perfect matching $M$ is called \dfn{$(r+1)$-noncrossing} if it does not contain~$r+1$ mutually crossing arcs. That is, the perfect matching $M$ is $(r+1)$-noncrossing if there do not exist pairs $(i_1,j_1), (i_2,j_2), \dots, (i_{r+1},j_{r+1}) \in M$ such that $i_1 < i_2 < \dots < i_{r+1} < j_1 < j_2 < \dots < j_{r+1}$. A $2$-noncrossing perfect matching is usually just called \dfn{noncrossing}.

There is a remarkable bijection due to Sundaram~\cite{sundaram1986combinatorics} that maps \dfn{$r$-symplectic oscillating tableaux} of weight zero to $(r+1)$-noncrossing perfect matchings.

We use the following partition notation. A partition is a weakly decreasing sequence $\lambda=(\lambda_1,\lambda_2,\dots)$ of nonnegative integers with finitely many nonzero parts; we omit trailing zeroes and identify $\lambda$ with its Young diagram. Thus $21$ denotes the partition $(2,1)$ and $11$ denotes $(1,1)$. For partitions $\alpha$ and $\beta$, we write $\alpha\subseteq\beta$ if $\alpha_i\leq \beta_i$ for all $i$, and define their intersection by
\[
(\alpha\cap\beta)_i=\min(\alpha_i,\beta_i).
\]
Let $\lambda'$ denote the conjugate partition, and let $e_k$ be the vector that adds one box to the $k$th part. Finally, for an integer vector $v$ with finitely many nonzero entries, $\operatorname{sort}(v)$ denotes the vector obtained by rearranging its entries in weakly decreasing order, with trailing zeroes omitted.

\begin{definition}[{\cite{sundaram1986combinatorics}}]
    An $r$-symplectic oscillating tableau of length $n$ and weight~$\lambda$ is a sequence of partitions $(\emptyset=\lambda^0,\lambda^1,\dots,\lambda^n=\lambda)$ such that
    \begin{enumerate}
        \item each partition has at most $r$ parts;
        \item for each $1\leq i\leq n$, the partition $\lambda^i$ is obtained from $\lambda^{i-1}$ by adding or removing one box.
    \end{enumerate}
    It has weight zero if $\lambda=\emptyset$.
\end{definition}

In type $C_r$, the crystal $C^{\mathrm{vec}}_r$ of the \dfn{standard representation} (i.e., \dfn{vector representation}) consists of the vertices $\{1,2,\dots,r,\overline{r},\dots,\overline{2},\overline{1}\}$. The $2n$th tensor power of this crystal consists of words $w_1w_2\dots w_{2n}$ in this alphabet.

The $r$-symplectic oscillating tableaux of weight zero and length $2n$ correspond to the highest weight elements of weight zero in the crystal $(C^{\mathrm{vec}}_r)^{\otimes 2n}$ by setting
\begin{equation}\label{eq:word-tableau-dictionary}
w_i = \begin{cases}
k&\text{if $\lambda^i=\lambda^{i-1}+e_k$},\\
\overline{k}&\text{if $\lambda^i=\lambda^{i-1}-e_k$}.
\end{cases}
\end{equation}

Promotion on $r$-symplectic oscillating tableaux is defined via \eqref{eq:crystal commutor} and can be computed with the same local rule as for fans of Dyck paths using \eqref{eq:promotion-diag}.

We now give Roby's growth-diagram description~\cite{Roby1991} of Sundaram's map for oscillating tableaux of weight zero.
\begin{definition}\label{def:sundaram-map}
Let $\mathcal{O}=(\emptyset=\lambda^0,\lambda^1,\dots,\lambda^n=\emptyset)$ be an $r$-symplectic oscillating tableau of weight zero. Define a triangular growth diagram with vertex labels $\gamma_{i,j}$ for $0\leq i\leq j\leq n$ and cells $\mathsf C_{i,j}$ for $1\leq i<j\leq n$. The cell $\mathsf C_{i,j}$ has corner labels
\[
\begin{array}{cc}
\gamma_{i-1,j-1} & \gamma_{i,j-1}\\[-2pt]
\gamma_{i-1,j} & \gamma_{i,j}
\end{array}.
\]
Set the diagonal labels to be the oscillating tableau:
\[
\gamma_{i,i}=\lambda^i \qquad (0\leq i\leq n),
\]
and set the first subdiagonal to be the common smaller partition:
\[
\gamma_{i,i+1}=\lambda^i\cap\lambda^{i+1}\qquad (0\leq i<n).
\]
Fill the remaining vertices and the cell contents recursively, in increasing order of $j-i$, using the following local rule. Suppose the three relevant already-filled corner labels are $\kappa=\gamma_{i-1,j-1},\mu=\gamma_{i,j-1},\nu=\gamma_{i,j}$ and the fourth corner label to be determined is $\lambda=\gamma_{i-1,j}$.
\begin{enumerate}
\item If $\kappa=\nu$ and $\mu=\kappa+e_1$, put a cross in the cell and set
\[
\lambda=\kappa.
\]
\item Otherwise, leave the cell empty and determine $\lambda$ by
\[
\lambda'=\operatorname{sort}(\kappa'+\nu'-\mu'),
\]
where missing parts of the conjugate partitions are treated as zero.
\end{enumerate}

Finally, define
\[
\mathcal{M}(\mathcal{O})
=
\bigl\{\{i,j\}:\text{ the cell }\mathsf C_{i,j}\text{ contains a cross}\bigr\}.
\]
\end{definition}

\begin{example}\label{ex:sundaram-growth}
    $\mathcal{O} = (\emptyset,1,11,21,2,1,\emptyset)$ is a $2$-symplectic oscillating tableau of weight zero. Its growth diagram is
    
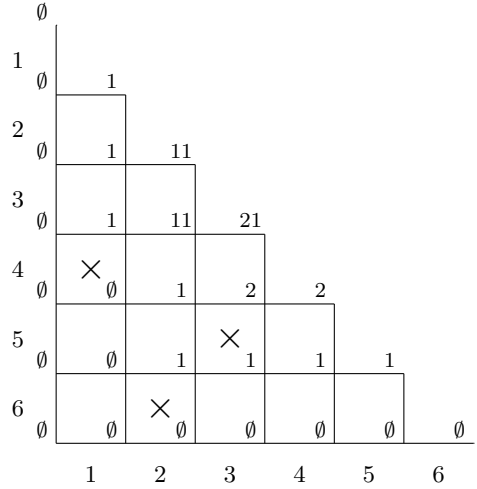
\begin{center}
\begin{tikzpicture}[scale=0.92]
\tikzset{part/.style={fill=white,inner sep=0.6pt,font=\scriptsize}}
\tikzset{idx/.style={font=\footnotesize}}
\foreach \y in {1,...,7}
   \draw (1,8-\y)--(\y,8-\y);
\foreach \x in {1,...,7}
   \draw (\x,8-\x)--(\x,1);
\foreach \x in {1,2,...,6}
   \node[idx] at (\x.5,0.55) {$\x$};
\foreach \y in {1,2,...,6}
   \node[idx] at (0.45,8-\y.5) {$\y$};
\node[font=\Large] at (1.5,8-4.5) {$\times$};
\node[font=\Large] at (3.5,8-5.5) {$\times$};
\node[font=\Large] at (2.5,8-6.5) {$\times$};
\node[part] at (.8,8-.8) {$\emptyset$};
\node[part] at (.8,8-1.8) {$\emptyset$};
\node[part] at (.8,8-2.8) {$\emptyset$};
\node[part] at (.8,8-3.8) {$\emptyset$};
\node[part] at (.8,8-4.8) {$\emptyset$};
\node[part] at (.8,8-5.8) {$\emptyset$};
\node[part] at (.8,8-6.8) {$\emptyset$};
\node[part] at (1.8,8-1.8) {$1$};
\node[part] at (1.8,8-2.8) {$1$};
\node[part] at (1.8,8-3.8) {$1$};
\node[part] at (1.8,8-4.8) {$\emptyset$};
\node[part] at (1.8,8-5.8) {$\emptyset$};
\node[part] at (1.8,8-6.8) {$\emptyset$};
\node[part] at (2.8,8-2.8) {$11$};
\node[part] at (2.8,8-3.8) {$11$};
\node[part] at (2.8,8-4.8) {$1$};
\node[part] at (2.8,8-5.8) {$1$};
\node[part] at (2.8,8-6.8) {$\emptyset$};
\node[part] at (3.8,8-3.8) {$21$};
\node[part] at (3.8,8-4.8) {$2$};
\node[part] at (3.8,8-5.8) {$1$};
\node[part] at (3.8,8-6.8) {$\emptyset$};
\node[part] at (4.8,8-4.8) {$2$};
\node[part] at (4.8,8-5.8) {$1$};
\node[part] at (4.8,8-6.8) {$\emptyset$};
\node[part] at (5.8,8-5.8) {$1$};
\node[part] at (5.8,8-6.8) {$\emptyset$};
\node[part] at (6.8,8-6.8) {$\emptyset$};
\end{tikzpicture}
\end{center}
We read off $\mathcal{M}(\mathcal{O})= \{\{1,4\},\{2,6\},\{3,5\}\}$.
\end{example}

Pfannerer--Rubey--Westbury~\cite{pfannerer2020promotion} have shown that Sundaram's bijection intertwines promotion and \dfn{rotation}. In other words, we have:

\begin{lemma}[{\cite{pfannerer2020promotion}}]\label{lem:map to matchings}
Let $C^{\mathrm{vec}}_r$ be the crystal for the standard representation in type~$C_r$. The map in \cref{def:sundaram-map} is a bijection from highest weight elements of weight zero in $(C^{\mathrm{vec}}_r)^{\otimes 2n}$ to $(r+1)$-noncrossing perfect matchings on $[2n]$ that intertwines promotion and rotation.
\end{lemma}

\subsection{\texorpdfstring{$2$}{2}-fans of Dyck paths and \texorpdfstring{$3$}{3}-noncrossing perfect matchings}

\newcommand{\twocell}[2]{
  \begin{tikzpicture}[scale=0.24]
    \draw[line width=0.7pt] (0,0) rectangle (1,2);
    \draw[line width=0.7pt] (0,1) -- (1,1);
    \node at (0.5,1.5) {\tiny $#1$};
    \node at (0.5,0.5) {\tiny $#2$};
  \end{tikzpicture}
}

The root systems of type $B_2$ and $C_2$ are isomorphic. Under this exceptional isomorphism, the spin representation of type $B_2$ and the standard representation of type~$C_2$ correspond and their crystals are isomorphic, exchanging the roles of $f_1$ and $f_2$. More precisely, see~\cref{fig:B2-C2-crystal-isomorphism}. The vertex map in this figure is
\[
\psi:\quad
\twocell{+}{+}\longmapsto 1,\qquad
\twocell{-}{+}\longmapsto 2,\qquad
\twocell{+}{-}\longmapsto \overline{2},\qquad
\twocell{-}{-}\longmapsto \overline{1}.
\]
Equivalently, $\psi$ is the crystal isomorphism sending the highest weight vertex of $B^{\mathrm{spin}}_2$ to the highest weight vertex of $C^{\mathrm{vec}}_2$ and exchanging the two simple-root directions.

Thus, combining \Cref{lem:promotion_crystal,lem:map to matchings} in this case $r=2$, we obtain:

\begin{thm}
\label{thm:promisrot}
There exists an explicit bijection between $2$-fans of Dyck paths of semilength $n$ and $3$-noncrossing perfect matchings on $[2n]$ that intertwines promotion and rotation.
\end{thm}

\begin{proof}
Let $\mathcal{F}=(\emptyset=\mu^0,\mu^1,\dots,\mu^{2n}=\emptyset)$ be a $2$-fan of Dyck paths.  As in \cref{lem:promotion_crystal}, form the tensor
\[
b(\mathcal{F})
=
(\mu^1-\mu^0)\otimes(\mu^2-\mu^1)\otimes\cdots\otimes(\mu^{2n}-\mu^{2n-1})
\in (B^{\mathrm{spin}}_2)^{\otimes 2n},
\]
where each difference is read as a vertex of the type $B_2$ spin crystal.  By \cref{lem:promotion_crystal}, this gives a bijection from $2$-fans of semilength $n$ to the highest weight elements of weight zero in $(B^{\mathrm{spin}}_2)^{\otimes 2n}$, and this bijection intertwines fan promotion with crystal promotion.

Apply $\psi$ tensor factor by tensor factor.  Since $\psi$ is the exceptional $B_2/C_2$ crystal isomorphism, with the two simple-root directions interchanged, the map
\[
\psi^{\otimes 2n}\colon (B^{\mathrm{spin}}_2)^{\otimes 2n}\longrightarrow (C^{\mathrm{vec}}_2)^{\otimes 2n}
\]
sends highest weight elements of weight zero bijectively to highest weight elements of weight zero.  Moreover, crystal promotion is defined from the crystal commutor~\eqref{eq:crystal commutor}, and the commutor is natural with respect to crystal isomorphisms; hence $\psi^{\otimes 2n}$ intertwines the type $B_2$ and type $C_2$ promotion operators.

Now interpret the resulting type $C_2$ word as a $2$-symplectic oscillating tableau by the dictionary preceding \cref{def:sundaram-map}; call this tableau $\mathcal{O}(\mathcal{F})$.  Finally apply Sundaram's map $\mathcal{M}$ from \cref{def:sundaram-map}.  By \cref{lem:map to matchings} with $r=2$, this last map is a bijection from the highest weight elements of weight zero in $(C^{\mathrm{vec}}_2)^{\otimes 2n}$ to $3$-noncrossing perfect matchings on $[2n]$ and intertwines crystal promotion with rotation.

Thus the desired bijection is the explicit composition
\[
\mathcal{F}
\longmapsto b(\mathcal{F})
\longmapsto \psi^{\otimes 2n}(b(\mathcal{F}))
\longmapsto \mathcal{O}(\mathcal{F})
\longmapsto \mathcal{M}(\mathcal{O}(\mathcal{F})),
\]
which is equivariant with respect to the relevant cyclic action because the bijections in this composition are each appropriately equivariant.
\end{proof}

\begin{figure}
\begin{center}
\begin{tikzpicture}[x=1cm,y=1cm]
\tikzset{cnode/.style={inner sep=2pt,font=\large}}

\node (S1) at (0,1.8) {\twocell{+}{+}};
\node (S2) at (3,1.8) {\twocell{-}{+}};
\node (S3) at (6,1.8) {\twocell{+}{-}};
\node (S4) at (9,1.8) {\twocell{-}{-}};

\node[cnode] (C1) at (0,0) {$1$};
\node[cnode] (C2) at (3,0) {$2$};
\node[cnode] (C3) at (6,0) {$\overline{2}$};
\node[cnode] (C4) at (9,0) {$\overline{1}$};

\node[anchor=east,font=\small] at (-.7,1.8) {$B^{\mathrm{spin}}_2$};
\node[anchor=east,font=\small] at (-.7,0) {$C^{\mathrm{vec}}_2$};

\draw[->, draw=red!60!black,ultra thick]  (S1) -- node[above] {\small 2} (S2);
\draw[->, draw=blue!60!black,ultra thick] (S2) -- node[above] {\small 1} (S3);
\draw[->, draw=red!60!black,ultra thick]  (S3) -- node[above] {\small 2} (S4);

\draw[->, draw=blue!60!black,ultra thick] (C1) -- node[below] {\small 1} (C2);
\draw[->, draw=red!60!black,ultra thick]  (C2) -- node[below] {\small 2} (C3);
\draw[->, draw=blue!60!black,ultra thick] (C3) -- node[below] {\small 1} (C4);

\draw[densely dotted,thick] (S1) -- (C1);
\draw[densely dotted,thick] (S2) -- (C2);
\draw[densely dotted,thick] (S3) -- (C3);
\draw[densely dotted,thick] (S4) -- (C4);
\node[font=\small] at (9.5,.9) {$\psi$};
\end{tikzpicture}
\end{center}
\caption{The exceptional isomorphism between the type $B_2$ spin crystal and the type $C_2$ vector crystal. The dotted lines give the vertex map $\psi$; red and blue arrows show that $f_2$ in type $B_2$ corresponds to $f_1$ in type $C_2$, and $f_1$ in type $B_2$ corresponds to $f_2$ in type $C_2$.}
\label{fig:B2-C2-crystal-isomorphism}
\end{figure}
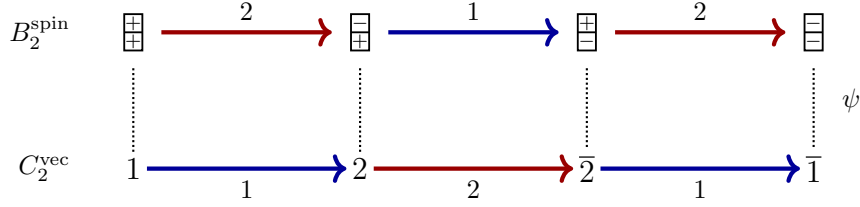

\begin{example} 
Consider the $2$-fan of Dyck paths
\[
\mathcal{F}=(00,11,20,31,22,11,00).
\]
Writing $\mathcal{F}=(\mu^0,\mu^1,\dots,\mu^6)$, its consecutive differences are
\begin{gather*}
\mu^1-\mu^0=(1,1),\quad
\mu^2-\mu^1=(1,-1),\quad
\mu^3-\mu^2=(1,1),\\
\mu^4-\mu^3=(-1,1),\quad
\mu^5-\mu^4=(-1,-1),\quad
\mu^6-\mu^5=(-1,-1).
\end{gather*}
With the convention that $(\epsilon_2,\epsilon_1)$ is displayed as
$\twocell{\epsilon_1}{\epsilon_2}$, where $+$ stands for $+1$ and $-$ stands for $-1$, this gives the spin-crystal tensor element
\[
b(\mathcal{F})
=
\twocell{+}{+}\otimes
\twocell{-}{+}\otimes
\twocell{+}{+}\otimes
\twocell{+}{-}\otimes
\twocell{-}{-}\otimes
\twocell{-}{-}
\in (B^{\mathrm{spin}}_2)^{\otimes 6}.
\]
Applying $\psi$ tensor factor by tensor factor gives the type $C_2$ word
\[
\psi(b(\mathcal{F}))
=
1\otimes 2\otimes 1\otimes \overline{2}\otimes \overline{1}\otimes \overline{1}
\in (C^{\mathrm{vec}}_2)^{\otimes 6}.
\]
Under the oscillating tableau dictionary for $C^{\mathrm{vec}}_2$, this word adds a box in part $1$, adds a box in part $2$, adds a box in part $1$, removes a box in part $2$, and then removes two boxes from part $1$. Hence it gives
\[
\mathcal{O}=(\emptyset,1,11,21,2,1,\emptyset),
\]
which is exactly the oscillating tableau in \cref{ex:sundaram-growth}. Therefore Sundaram's bijection sends $\mathcal{F}$ to
\[
\mathcal{M}(\mathcal{O})=\{\{1,4\},\{2,6\},\{3,5\}\},
\]
that is, to the $3$-noncrossing perfect matching
\begin{tikzpicture}[line width=1pt,baseline=(current bounding box.center)]
    \node (a) [draw=none, minimum size=0.9cm, regular polygon, regular polygon sides=6] at (0,0) {};
    \foreach \n [count=\i from 0, remember=\n as \lastn, evaluate={\i+\lastn}] in {1,2,...,6}
        \path (a.center) -- (a.corner \n) node[pos=1.5] {$\n$};
    \draw (a.corner 1) to (a.corner 4);
    \draw (a.corner 2) to (a.corner 6);
    \draw (a.corner 3) to (a.corner 5);
\end{tikzpicture}. Thus, this is the image of $\mathcal{F}$ under the bijection from \cref{thm:promisrot}.
\end{example}

\begin{remark}
As mentioned, we believe that the bijection from~\cref{thm:promisrot} intertwining promotion of $2$-fans of Dyck paths and rotation of $3$-noncrossing perfect matchings is the same as that constructed by Chen et al.~in~\cite{chen2007crossings}, although we will not prove that here.
\end{remark}

\section{Electrical networks and canonical bases} \label{sec:electrical}

\subsection{The Grassmannian and plabic networks}

The \dfn{(complex) Grassmannian} $\mathrm{Gr}(k,n)$ is the space of all $k$-dimensional subspaces of $\mathbb{C}^n$. The Grassmannian can be embedded into projective space via \dfn{Pl\"ucker coordinates}. If $W$ is the row space of a $k\times n$ matrix, then the Pl\"ucker coordinates of $W$ are the $k\times k$ minors of that matrix. Choosing a different matrix to represent $W$ has the effect of multiplying all minors by the same constant, so the Pl\"ucker coordinates determine a point of~$\mathbb{P}^{\binom{n}{k}-1}$. The image of this embedding is a subvariety and the equations which cut it out are called the \dfn{Pl\"ucker relations}.

Postnikov~\cite{postnikov2006total} introduced plabic networks to study the \dfn{totally nonnegative Grassmannian} $\mathrm{Gr}(k,n)_{\geq 0}$, the subset of the Grassmannian for which all Pl\"ucker coordinates are nonnegative real numbers. A \dfn{plabic network} is a planar graph embedded in a disk with positive real edge weights and a bicoloring of its vertices. For each size-$k$ subset $I$ of $[n]$ and each plabic network $N$ with $n$ boundary vertices Postnikov defines the \dfn{boundary measurement} $\Delta_I(N)$ and shows that they satisfy the Pl\"ucker relations. Thus, each plabic network defines a point in the totally nonnegative Grassmannian.

\subsection{The space of electrical networks}

A \dfn{circular planar electrical network} $e(\Gamma, \omega)$ is a planar graph $\Gamma$ embedded in a disk with edge weights $\omega$. If the edge weights are interpreted as resistances, then applying a voltage to the boundary vertices induces a current in the boundary vertices in accordance with Kirchhoff's laws. The matrix $M_R(e)$ encoding the linear map from input voltages to currents is the {\em response matrix} of the electrical network $e$. Two electrical networks are electrically equivalent if they have the same response matrix, and \dfn{the space of electrical networks}~$E_n$ is the set of all electrical networks modulo electrical equivalence. 

Lam~\cite{lam2018electroid} uses plabic networks to construct an embedding from the space of electrical networks to the totally nonnegative Grassmannian $ \mathrm{Gr}(n-1,2n)_{\geq 0}$.
Using a version of the generalized Temperley's trick, Lam defines a map $N$ from electrical networks with $n$ boundary vertices to plabic networks with $2n$ boundary vertices such that electrically equivalent networks are sent to move-equivalent plabic networks. As each plabic network represents a point in $ \mathrm{Gr}(n-1,2n)_{\geq 0}$, he obtains a map 
\[
\iota: E_n \hookrightarrow \mathrm{Gr}(n-1,2n)_{\geq 0}.
\]
He then proves this map is injective, and defines \dfn{the electroid variety} $\mathcal{X}_n$ to be the Zariski closure of its image.

\begin{remark}
    The space of electrical networks is not compact, so Lam~\cite{lam2018electroid} introduces ``cactus networks'' to remedy this. We only make use of the Zariski closure which is the same for either space, so we stick to electrical networks for simplicity.
\end{remark}

\subsection{The Lagrangian Grassmannian}
Recall that if $V$ is a $2r$-dimensional complex vector space with a symplectic form $\Lambda$, the \dfn{Lagrangian Grassmannian} $\mathrm{LG}(r,V)$ is the space of all maximal (hence, $r$-dimensional) isotropic (with respect to $\Lambda$) subspaces of $V$. We also denote this Lagrangian Grassmannian by $\mathrm{LG}(r,2r)$.

Bychkov, Gorbounov, Kazakov, and Talalaev \cite{bychkov2023electrical} (see also~\cite{chepuri2026electrical}) connect Lam's results to the representation theory of the symplectic group and the Lagrangian Grassmannian. In particular, Bychkov et al.~show that Lam's embedding can be alternatively defined as follows. Given an electrical network $e \in E_n$, let its response matrix be $M_R(e) = (x_{ij})$ and define the matrix \[
\Omega_n(e) = \begin{bmatrix}
    x_{11} & 1 & -x_{12} & 0 & x_{13} & 0 &\cdots & (-1)^n\\
    -x_{21} & 1 & x_{22} & 1 & -x_{23} & 0&\cdots & 0\\
    x_{31} & 0 & -x_{32} & 1 & x_{33} & 1&\cdots & 0\\
    \vdots &\vdots &\vdots\vdots &\vdots &\vdots&\vdots&\ddots &\vdots
\end{bmatrix}
\]
Then $\iota(e)$ is the row span of $\Omega_n(e)$.
Using this characterization of the embedding~$\iota$, Bychkov, Gorbounov, Kazakov, and Talalaev show that the image is contained inside the Lagrangian Grassmannian $\mathrm{LG}(n-1, 2n-2)$ for a specific symplectic form. Specifically, for $1\leq i\leq 2n-2$ let $w_i$ denote the $2n$-dimensional vector in~$\mathbb{C}^{2n}$ with a $1$ in position $i$ and $i+2$ and $0$'s elsewhere, e.g. $w_1 = (1,0,1,0,0,\cdots,0)$. Let~$V$ be the $(2n-2)$-dimensional subspace of $\mathbb{C}^{2n}$ which is the span of the $w_i$ and let~$\beta$ be the basis $\{w_i\}_{1\leq i\leq 2n-2}$. Then, Bychkov et al.~\cite{bychkov2023electrical} prove the following theorem.

\begin{thm}[{\cite{bychkov2023electrical}}]
\label{bychkovisotropic}
    Let $e$ be an electrical network. Then $\iota(e) \subseteq V$ and is isotropic with respect to the symplectic form 
\[
\Lambda_{2n-2} = \begin{bmatrix}
    0 & 1& 0& 0& 0&\cdots& 0\\
-1 &0& -1 &0 &0&\cdots& 0\\
0 &1& 0 &1 &0&\cdots &0\\
0&0&-1&0&-1&\cdots &0\\
\vdots &\vdots  &\vdots&\vdots&\vdots&\ddots &\vdots
\end{bmatrix}
\]
where $\Lambda_{2n-2}$ is expressed in the $\beta$ basis. Furthermore, $\mathcal{X}_n$ is isomorphic to the Lagrangian Grassmannian $\mathrm{LG}(n-1, V)$ for this symplectic form.
\end{thm}

\subsection{The cyclic action on the space of electrical networks}

There is a natural cyclic action of order $2n$ on $\mathcal{X}_n$, which is easiest to describe using the identification of~$\mathcal{X}_n$ with $\mathrm{LG}(n-1, V)$ from~\cref{bychkovisotropic}. Namely, let $c\colon \mathbb{C}^{2n} \to \mathbb{C}^{2n}$ be the following map:
\[
c \cdot (x_1, x_2, \dots, x_{2n}) := (x_2, \cdots, x_{2n}, (-1)^{n}x_1)
\]
This $c$ induces an action of $\mathbb{Z}/2n\mathbb{Z}$ on the Lagrangian Grassmannian~$\mathrm{LG}(n-1,V)$, hence under the identification in \cref{bychkovisotropic}, on the electroid variety $\mathcal{X}_n$. In the electrical networks picture, its square $c^2\colon \mathcal{X}_n\to\mathcal{X}_n$ corresponds to rotation of the labels of the electrical network.

\subsection{The grove algebra and the bush basis}

Gao, Lam, and Xu~\cite{gao2025electrical} call the homogeneous coordinate ring of the electroid variety $\mathcal{X}_n$ the \dfn{grove algebra}, denoted~$G_n$, with its degree $m$ piece denoted by~$G_{m,n}$. The grove algebra is governed by the combinatorics of noncrossing set partitions.

Recall that a set partition $\sigma$ of $[n] := \{1,2,\ldots,n\}$ is \dfn{noncrossing} if there do not exist $a<b<c<d$ such that $a,c$ lie in one block of $\sigma$ and $b,d$ lie in a different block of $\sigma$.  Equivalently, if $1,\dots,n$ are placed clockwise on a circle and the elements in each block are joined inside the disk, the convex hulls of distinct blocks do not cross.  We write $\mathcal{NP}(n)$ for the set of noncrossing set partitions of $[n]$. We will also use the Kreweras complement.  Given $\sigma\in\mathcal{NP}(n)$, place
\[
1,\overline{1},2,\overline{2},\dots,n,\overline{n}
\]
clockwise on a circle.  Draw the blocks of $\sigma$ on the unbarred vertices.  There is a unique coarsest noncrossing partition of the barred vertices whose union with $\sigma$ is still noncrossing in this $2n$-point circle; after identifying $\overline{i}$ with $i$, this partition is the \dfn{Kreweras complement} $\mathrm{Krew}(\sigma)$.  The Kreweras complement has order $2n$ on $\mathcal{NP}(n)$, with its square equal to cyclic rotation of the labels. There is a natural bijection between noncrossing partitions of $[n]$ and noncrossing perfect matchings on~$[2n]$ under which Kreweras complement of the set partition corresponds to rotation of the perfect matching.

The grove algebra $G_n$ is generated in degree one by \dfn{grove coordinates} $L_\sigma$ indexed by $\mathcal{NP}(n)$. Grove coordinates can be thought of as analogues of Pl\"ucker coordinates for the grove algebra, and Lam~\cite{lam2018electroid} shows that these two coordinate systems are related by the following formula.

\begin{thm}[{Cf.~\cite[Theorem 5.10]{lam2018electroid}}]
    For any electrical network $e$ and $I \in \binom{[2n]}{n-1}$, we have
    \[
    \Delta_I(N(e)) = \sum_{\substack{\sigma \in \mathcal{NP}(n)\\\sigma \textrm{ concordant with } I}} L_{\sigma}(e).
    \]
\end{thm}

We omit the definition of concordance but note that concordance is preserved under a simultaneous cyclic shift of $I$ and Kreweras complement of $\sigma$. Thus, the cyclic $c\colon \mathcal{X}_n\to\mathcal{X}_n$ permutes the grove coordinates according to Kreweras complement of the indexing noncrossing set partitions (equivalently, rotation of the indexing noncrossing perfect matchings). 

Gao, Lam, and Xu~\cite{gao2025electrical} conjecture the existence of an \dfn{electrical canonical basis} of~$G_{m,n}$ with the following properties.

\begin{conj}[{Cf.~\cite[Conjecture~1.4]{gao2025electrical}}] \label{conj:gao}
For all $m \geq 1$, $G_{m,n}$ has an electrical canonical basis which, among other properties, satisfies the following:
\begin{itemize}
\item its elements are indexed by height $m$ staircase plane partitions of size $n-1$;
\item the induced action of $\langle c \rangle\simeq \mathbb{Z}/2n\mathbb{Z}$ on $G_{m,n}$ corresponds to promotion of staircase plane partitions on these basis elements.
\end{itemize}
\end{conj}

We now explain the relationship of \cref{conj:gao} to the cyclic sieving conjecture, \cref{conj:main}.

\begin{lemma} \label{lem:char_comp}
If \cref{conj:gao} is true for some $n, m \geq 1$, then \cref{conj:main} is true for that same $m$ and for $P=\delta_{n-1}$ a staircase of size $n-1$.
\end{lemma}

\begin{proof}
The cases $n=1,2$ are straightforward, so let us assume for convenience that~$n \geq 3$ so that the desired cyclic sieving polynomial is given by~\eqref{eqn:q-multi-cat}. Because of~\cref{thm:striker_williams}, it is equivalent to prove the cyclic sieving result for promotion instead of rowmotion.

Since the conjectured electrical canonical basis is permuted by this cyclic action $\langle c\rangle \simeq \mathbb{Z}/2n\mathbb{Z}$ according to promotion, the trace of $c^k$ acting on $G_{m,n}$ is equal to the number of basis elements fixed by $c^k$, i.e., the number of plane partitions in $\mathcal{PP}^m(\delta_{n-1})$ fixed by $\mathrm{Pr}^{k}$. But the trace of $c^k$ can also be computed through the character theory of the symplectic group. In the $\beta$ basis, $c$ can be expressed as
\[
c = \begin{bmatrix}
    0 & 1 & 0 & -1 & \cdots & 0& (-1)^{n}\\
    1 & 0 & 0 &0&\cdots&\vdots&0\\
    0 & 1 & 0 &0&\cdots &\vdots&0\\
    0& 0 & 1 &0&\cdots &\vdots&0\\
    \vdots & \vdots &\vdots &\vdots &\ddots & \vdots& \vdots\\
    0 & 0 & 0 &0&\cdots& 1 & 0&\\
\end{bmatrix}
\]
It is simple to check that $c\Lambda_{2n-2} c^T = -\Lambda_{2n-2}$. Thus, $ic \in \mathrm{Sp}(2n-2, \mathbb{C})$, the symplectic group. Cofactor expansion along the top row of $ic-\lambda I_{2n-2}$ gives a characteristic polynomial of $\frac{\lambda^{2n}-1}{\lambda^2-1}$ and thus $ic$ has eigenvalues 
\[\zeta^{\pm 1}, \zeta^{\pm 2}, \dots, \zeta^{\pm (n-1)}\]
where $\zeta = e^{\frac{2\pi i}{2n}}$. The lift of scalar multiplication by $i$ to $G_{m,n}$ has trace $i^{-m(n-1)}$, so the trace of $c^k$ acting on $G_{m,n}$ is $i^{km(n-1)}$ multiplied by the symplectic character specialization $\mathrm{Sp}_{2n-2}(m\omega_{n-1} ; \zeta^k, \zeta^{2k}, \dots, \zeta^{(n-1)k})$. By a result of Proctor~\cite[Theorem 1, case (CGI)]{proctor1990symmetric} (see also~\cite{proctor1988odd}), we have
\begin{align*}
\mathrm{Sp}_{2n-2}(m\omega_{n-1} ; q,q^2,q^3,\dots,q^{n-1}) &= q^{-m\binom{n}{2}}\prod_{1\leq i \leq j \leq n-1} \frac{ 1-q^{i+j+2m}}{1-q^{i+j}}
\end{align*}
and thus
\[
\mathrm{tr}(c^k) = i^{km(n-1)}\zeta^{-km\binom{n}{2}}\prod_{1\leq i \leq j \leq n-1} \frac{ 1-\zeta^{k(i+j+2m)}}{1-\zeta^{k(i+j)}} = \Omega_{\delta_{n-1}}(m;\zeta^k)
\]
as desired.
\end{proof}

In the case $m=1$ the basis of grove coordinates will satisfy \cref{conj:gao}. For the case $m=2$, Gao--Lam--Xu~\cite{gao2025electrical} introduced the \dfn{bush basis} $\{B_\xi\}_{\xi \in \mathcal{TC}(2n)}$ of~$G_{2,n}$ indexed by the set $\mathcal{TC}(2n)$ of $3$-noncrossing perfect matchings of $[2n]$. It is similar to the basis of Temperley--Lieb immanants for the degree two part of the usual Grassmannian. They give a combinatorial formula for the expansion of degree two grove monomials into the bush basis.

\begin{thm}[{Cf.~\cite[Theorem~4.4]{gao2025electrical}}] \label{thm:bushbasisexpansion}
For $\sigma,\sigma' \in \mathcal{NP}(n)$, and any electrical network $e$,
\[
L_\sigma(e)L_{\sigma'}(e) = \sum_{\xi \in \mathcal{TC}(2n)}  a_{\sigma,\sigma',\xi }B_\xi(e)
\]
where $a_{\sigma,\sigma',\xi}$ counts valid opposite loopless resolutions of $\xi$ resulting in $\sigma,\sigma'$.
\end{thm}

For our purposes, we do not need the precise definition of a valid opposite loopless resolution except to note that simultaneously rotating $\xi$ and applying Kreweras complement to $\sigma$ and to $\sigma'$ preserves their number. Thus, the induced cyclic action of $\langle c \rangle\simeq \mathbb{Z}/2n\mathbb{Z}$ on $G_{2,n}$ corresponds in the bush basis to rotation of the indexing $3$-noncrossing perfect matchings. Hence, by combining \cref{thm:promisrot} and \cref{lem:pp_dyck_bij} with these observations, we conclude:

\begin{lemma} \label{lem:bush}
The bush basis of $G_{2,n}$ satisfies the conditions of \cref{conj:gao}.
\end{lemma}

Of course, \cref{lem:char_comp,lem:bush} together imply \cref{thm:main}, completing the proof of our main result.

\begin{remark}
From the above, it follows that (for $n \geq 2$) the order of $\mathrm{Row}$ acting on~$\mathcal{PP}^2(\delta_n)$ is~$2(n+1)$. In fact, as suggested by \cref{conj:main}, it should be that, for all $m \geq 1$, $\mathrm{Row}^{n+1}$ acts on $\mathcal{PP}^m(\delta_n)$ as the order two automorphism, i.e., transposition. But because we work with promotion and not directly with rowmotion, we cannot immediately conclude this, even in this case $m=2$.
\end{remark}

\section{Future directions} \label{sec:future}

We conclude with a brief discussion of some future directions. 

\subsection{Electrical canonical basis in all degrees}
Of course, the most enticing future direction would be to try to define an electrical canonical basis satisfying \cref{conj:gao} in all degrees. We do not have any concrete proposal for that, although the development of spin webs in~\cite{bodish2024spin} may provide a first step.

\subsection{Type B/C duality} 
The exceptional isomorphism between $B_2$ and $C_2$ allowed us to show, using the crystal perspective, that promotion of staircase plane partitions of height two is equivalent to rotation of $3$-noncrossing perfect matchings. It is not true that promotion of higher height staircase plane partitions is equivalent to rotation of higher noncrossing perfect matchings. Indeed, the numbers of $(k+1)$-noncrossing perfect matchings and staircase plane partitions of height $k$ are not the same for $k > 2$.\footnote{In fact, unlike for staircase plane partitions, there is no product formula at all for the number of $(k+1)$-noncrossing perfect matchings for $k > 2$.} Nevertheless, there is still a mysterious ``duality'' between type $B$ and type $C$ related to staircase plane partitions of higher heights. Namely, these plane partitions arise as indexing sets both for a basis of the space of invariants of a tensor power of spin representations (a type~$B$ thing), and for a basis of a component of the coordinate ring of the Lagrangian Grassmannian (a type~$C$ thing). Bruce Westbury previously asked if this numerical coincidence had any deeper algebraic meaning in a MathOverflow question~\cite{westbury2011symplectic}. Understanding this duality better might help with defining the electrical canonical basis. Is this an avatar of Howe duality?

\subsection{Other homogeneous varieties}
Finally, we note that the general paradigm we have followed here, where we realize a cyclic action on a combinatorial set as an appropriate cyclic action on a nice basis of the coordinate ring of a homogeneous variety, could possibly also be followed to resolve other cases of \cref{conj:main}. Standard Monomial Theory~\cite{seshadri2014intro} explains how $P$-partitions for various posets $P$ often index bases of such coordinate rings. In particular, shifted staircase plane partitions index a basis of the coordinate ring of the (maximal) orthogonal Grassmannian. But, just as we needed here the ``right version'' of the Lagrangian Grassmannian (namely, the space of electrical networks), it might also be important to choose the right version of the orthogonal Grassmannian. A natural candidate is the Ising model as studied by Galashin and Pylyavskyy~\cite{galashin2020ising}, which also gives a positive part to the orthogonal Grassmannian. Indeed, we believe this paradigm implicitly relies on the important, but somewhat mysterious, connection between cyclic symmetry and total positivity (see, e.g.,~\cite{karp2019moment}).

\bibliography{main}{}
\bibliographystyle{abbrv}

\end{document}